\documentclass[10pt]{article}

\usepackage{amsfonts}
\usepackage{amssymb}
\usepackage[centertags]{amsmath}
\usepackage{amsthm}
\usepackage{graphics,graphicx}
\usepackage[margin=3cm]{geometry}
\usepackage{dsfont}
\usepackage{bm}
\usepackage{graphics,graphicx}
\usepackage{subfigure}
\usepackage{amsmath}
\usepackage[all]{xy}
\usepackage{color}

\newtheorem{thm}{Theorem}[section]
\newtheorem{prop}[thm]{Proposition}
\newtheorem{lem}[thm]{Lemma}
\newtheorem{lemma}[thm]{Lemma}
\newtheorem{cor}[thm]{Corollary}
\newtheorem{cl}[thm]{Claim}
\newtheorem{dd}[thm]{Definition}
\newtheorem{definition}[thm]{Definition}

\newtheorem{rmk}[thm]{Remark}
\newtheorem{remark}[thm]{Remark}

\newtheorem*{thm*}{Theorem}
\newtheorem*{cor*}{Corollary}
\newtheorem*{prop*}{Proposition}

\newtheorem*{bartwistedcomplex}{Proposition \ref{th:bartwistedcomplex}}
\newtheorem*{rankmasterinequalities}{Proposition \ref{th:rankmasterinequalities}}

\newcommand{\e}{\epsilon}
\newcommand{\Z}{\mathbb{Z}}
\newcommand{\A}{\mathcal{A}}
\renewcommand{\d}{\delta}
\newcommand{\J}{\mathcal{J}}
\renewcommand{\H}{\mathcal{H}}
\newcommand{\R}{\mathbb{R}}
\newcommand{\CC}{\mathcal{C}}
\newcommand{\C}{\mathbb{C}}
\newcommand{\M}{\mathcal{M}}
\newcommand{\D}{\mathcal{D}}
\newcommand{\Fuk}{\mathcal{F}uk}
\newcommand{\F}{\mathcal{F}}
\newcommand{\Diff}{\mathit{Diff}}
\newcommand{\fib}{\mathit{fib}}

\newcommand{\Lk}{\mathcal{L}_k}
\newcommand{\s}{\sigma}

\newcommand{\bdm}{\begin{displaymath}}
\newcommand{\edm}{\end{displaymath}}

\newcommand{\bq}{\begin{equation}}
\newcommand{\eq}{\end{equation}}

\numberwithin{equation}{section}

\title{Dehn twists and free subgroups  \\ of symplectic mapping class groups}
\author{Ailsa Keating}

\begin{document}
\maketitle

\begin{abstract}
Given two Lagrangian spheres in an exact symplectic manifold, we find conditions under which the Dehn twists about them generate a free non-abelian subgroup of the symplectic mapping class group. This extends a result of Ishida for Riemann surfaces \cite{Ishida}. The proof generalises the categorical version of Seidel's long exact sequence \cite{Seidel03} to arbitrary powers of a fixed Dehn twist. We also show that the Milnor fibre of any isolated degenerate hypersurface singularity contains such pairs of spheres. 
\end{abstract}

\tableofcontents

\section{Introduction}

Given any Lagrangian sphere $L$ in a symplectic manifold $M$, one can define the Dehn twist $\tau_L$ about it. This  is a symplectomorphism supported in a tubular neighbourhood of $L$, that is defined up to Hamiltonian isotopy, and a choice of framing for the sphere (\cite{Seidel99}). If $M$ and $L$ are exact, so is the symplectomorphism. These are sometimes known as generalised Dehn twists; in real dimension two, they recover the classical notion of a Dehn twist on a Riemann surface (discussed e.g.~in \cite{FLP}).
They arise, for instance, when studying the monodromy of a Milnor fibration, thought of as a compactly supported symplectomorphism. For a Morse singularity, this automorphism is a Dehn twist. For a general isolated hypersurface singularity, it can be expressed as a composition of Dehn twists about a collection of Lagrangian spheres. 
A natural problem is to study the subgroup of symplectomorphisms (up to some isotopy) generated by these. More generally, given a symplectic manifold with a set of distinguished Lagrangian spheres, one might want to understand the automorphisms generated by the Dehn twists about them;
 we study a sub-question.

Fix two Lagrangian spheres, say $L$ and $L'$. If they are disjoint, the two Dehn twists commute. If  they have exactly one intersection point, the Dehn twists have a braid relation: 
$\tau_L \tau_{L'} \tau_L = \tau_{L'} \tau_L \tau_{L'} $. (One can check this manually for Riemann surfaces. In general, it can be derived from monodromy computations for $(A_2)$ type singularities; see \cite{SeidelThomas} and Appendix A of \cite{Seidel99}.) 
Generically, it seems there need not be any other relations \cite{BirmanHilden, KhovanovSeidel}. We consider the case where the two spheres intersect twice or more; Ishida has studied this for Riemann surfaces.

\begin{thm*}[Ishida \cite{Ishida}]
Suppose $a$, $b$ are a pair of simple closed curves on $\Sigma_{g,n}$, an $n$-punctured genus $g$ surface, and that $I_{min}(a,b) \geq 2$. Then there are no relations between $\tau_a$ and $\tau_b$.
\end{thm*}
$I_{min}(a,b)$ is the minimal intersection number of $a$ and $b$, as defined in \cite{FLP}; it is the smallest unsigned count of intersection points, varying over the isotopy classes of $a$ and $b$. The key technical lemma for Ishida's result is the inequality
\bq
|n| \cdot I_{min}(c,a) \cdot I_{min}(a,b)
\leq I_{min}(c,b) + I_{min}(\tau^n_a (c), b) 
\eq
that holds for any triple of simple closed curves $a$, $b$ and $c$. 
To generalise this result to higher dimensions, we need to find a suitable measure of `intersection'. Suppose that $M$ is an exact symplectic manifold with contact type boundary, $L$ a Lagrangian sphere, and $L_0$, $L_1$ arbitrary compact exact Lagrangians. Seidel's long exact sequence \cite{Seidel03} gives the inequality of Floer ranks
\bq
dim (HF(L_1, L)) \cdot dim (HF(L, L_0)) \leq dim (HF(L_1, L_0)) + dim ( HF( \tau_{L} (L_1), L_0).
\eq
This suggests using rank of Floer cohomology groups as a substitute for intersection numbers; these recover the minimal intersection number of any two non-isotopic simple closed curves. However, for any exact Lagrangian sphere $L$ and any Hamiltonian perturbation $\phi$, $rk (HF(L, \phi(L))) = 2$, and Dehn twists about these spheres are isotopic. 
We have to make sure to exclude such cases. Such a pair $L$ and $\phi(L)$ are Lagrangian isotopic, which in particular implies that they are quasi-isomorphic objects of the Fukaya categeory of $M$, which in this case just means that the Floer products
\begin{eqnarray}
HF(L, \phi(L)) \otimes HF(\phi(L), L) & \to & HF(\phi(L), \phi(L)) \\
HF(\phi(L), L) \otimes HF(L, \phi(L)) & \to & HF(L, L)
\end{eqnarray}
are surjective (see Corollary \ref{th:surjectivegiveiso}). Our main result is:

\begin{thm}\label{th:freegroup}
Suppose $n \geq 2$. Let $M^{2n}$ be an exact symplectic manifold with contact type boundary, and  $\pi_0(Symp(M))$ the group of symplectomorphisms of $M$, up to symplectic isotopy.
Suppose $L$ and $L'$ are two Lagrangian spheres such that $dim (HF(L,L')) \geq 2$; additionally, if $dim(HF(L,L'))=2$, we require that $L$ and $L'$ be not quasi-isomorphic in the Fukaya category of $M$. Then the Dehn twists $\tau_L$ and $\tau_{L'}$ generate a free subgroup of $\pi_0(Symp(M))$.
\end{thm}
A Dehn twist induces a functor from the Fukaya category of $M$, $\Fuk(M)$, to itself, defined up to quasi-isomorphism;  it is invertible, also up to quasi-isomorphism. In particular, the induced functor on the cohomology category $H(\Fuk(M))$ is invertible.
We use a stronger form of the long exact sequence of \cite{Seidel03} that is phrased using twisted complexes, viewing the Dehn twist as such a functor \cite[Corollary 17.17]{Seidel08}.  The proof of Theorem \ref{th:freegroup} hinges on a generalisation of this to arbitrary powers of a fixed Dehn twist (Proposition \ref{th:bartwistedcomplex}). 
 As a by-product we show that:
 
 \begin{thm}\label{th:freegrouponFukaya}
 $\tau_L$ and $\tau_{L'}$ generate a free subgroup of automorphisms of $H(\Fuk(M))$. 
 \end{thm}

\begin{rmk} The hypotheses on $M$, $L$ and $L'$ under which we prove Theorem \ref{th:freegroup} are actually weaker, but more technical.  Note that if $L$ and $L'$ are Lagrangian isotopic, they are quasi-isomorphic in the Fukaya category. As we heavily use Floer and Fukaya--theoretic tools, which cannot tell quasi-isomorphic elements apart, it seems unlikely that we will be able to relax that condition without drastically modifying our approach.
\end{rmk}

We give examples of such pairs of Lagrangian spheres in the Milnor fibres of all degenerate isolated hypersurface singularities. For some of these, we also arrange for the spheres to be homologous, in which case the result cannot be at all detected using the Picard-Lefschetz theorem, or other homological tools. Such examples exist for all degenerate isolated singularities in the even complex-dimensional case. In the odd-dimensional case, though we cannot find any for $(A_2)$, they exist at least for all singularities adjacent to $(A_3)$ (see section \ref{sec:adjacency} for a definition).

Suppose $M$ satisfies the hypothesis of Theorem \ref{th:freegroup}. Let $\pi_0(\Diff^+M)$ be the group of orientation-preserving diffeomorphisms of $M$, up to smooth isotopy. There is a natural map 
\begin{equation}
\pi_0(Symp(M)) \to \pi_0 (\Diff^+M).
\end{equation}
What can we say about its kernel? Whenever $dim_{\R} M = 4$, one has $\tau^2_L = Id \in \pi_0 (\Diff^+M)$ \cite[Lemma 6.3]{Seidel99}. Moreover, whenever $dim_{\R} M = 2n$ for $n$ even, we have that $\tau^{k_n}_L = Id \in  \pi_0 (\Diff^+M) $, for some integer $k_n$ \cite[Section 3]{Krylov}. We then have the following immediate corollary to Theorem \ref{th:freegroup}.

\begin{cor}
Suppose $M^{2n}$ satisfies the hypothesis of Theorem \ref{th:freegroup}, and $n$ is even. Then the kernel of the forgetful map $\pi_0(Symp(M)) \to \pi_0 (\Diff^+M)$  contains a free non-abelian subgroup, generated by $\tau^{k_n}_L$ and $\tau^{k_n}_{L'}$.
\end{cor}

Additionally, for all the examples of homologous $L$, $L'$ that we construct in Milnor fibres, we show either that they are smoothly isotopic (even-dimensional case), or, in the odd complex-dimensional case, that they can be made disjoint after a smooth isotopy. Thus $\tau_L$ and $\tau_{L'}$ commute (or even agree) as elements of $\pi_0 (\Diff^+M)$, and the kernel again has a free non-abelian subgroup.

\section{Outline}

The main calculation involves twisted complexes associated to the Fukaya category of $M$, an $A_\infty$ category. Accordingly, section \ref{sec:Ainfinitypreliminaries} collects some background material on $A_\infty$ categories, and twisted complexes thereof. Section \ref{sec:Fukayapreliminaries} does the same for the Fukaya category of an exact symplectic manifold with contact type boundary. The first noteworthy ingredient is:

\begin{itemize}
\item  Let $L$ a Lagrangian sphere, and $A$ be the $A_\infty$ algebra on two generators given by $\Z_2[\e]/(\e^2)$, with $\mu^2(1,\e) = \mu^2(\e,1) = \e$, and all other $A_\infty$ structure maps trivial. 
 The Floer chain group $CF(L, L)$ with its Fukaya $A_\infty$ structure is quasi-isomorphic to $A$. (Proposition \ref{th:lagspherestructure}.)

\end{itemize}
This enables us to obtain two technical criteria for quasi-isomorphisms of Lagrangians, both of which feed into the main proof:
\begin{itemize}

\item
Suppose that $L_0$ and $L_1$ are Lagrangian spheres, that $dim (HF(L_0, L_1))=2$, and that multiplication 
\bdm
HF(L_0,L_1) \otimes HF(L_1, L_0) \to HF(L_1, L_1)
\edm
is surjective. Then $L_0$ and $L_1$ are quasi-isomorphic in the Fukaya category. 
(Corollary \ref{th:surjectivegiveiso}.)

\item 
Take two Lagrangian spheres $L_0$, $L_1$. Let $\epsilon \in HF(L_0,L_0) \cong H^\ast(L_0;\Z/2)$ be the generator corresponding to the top degree cohomology class. Consider its action by product on $HF(L_0,L_1)$, which is the map
\begin{align}
HF (L_0, L_1) & \to HF (L_0, L_1)  \\
[a] & \mapsto [\mu_2(a,\e)].
\end{align}
If it is nonzero, then $L_0$ and $L_1$ are quasi-isomorphic in the Fukaya category. (Corollary \ref{th:multiplybyenontrivialimage}.)

\end{itemize}
The main body of the proof is in sections \ref{sec:generalisation} and \ref{sec:lowerbounds}. First, some notation: let $L$ be a Lagrangian sphere, and $L_0$, $L_1$ any exact Lagrangians. $hom(L, L_0)$ denotes morphisms in the Fukaya category, and $ev$ evaluations maps, for instance $hom(L, L_0) \otimes L \to L_0$. 

Let $(\e)$ be the one-dimensional $\Z_2$--vector space generated by $\e$. 
Iterated applications of the long exact sequence of \cite{Seidel03} give an expression for $\tau^n_L L_0$; we start by simplifying this to show that:

\begin{bartwistedcomplex}
$\tau^n_L L_0$
 it is quasi-isomorphic to the twisted complex
\bq \label{eq:twistedcomplexintro}
 L_0 \; \oplus \; hom(L,L_0)\otimes L \; \oplus \; hom(L,L_0)  \otimes  (\e) \otimes L \;\oplus \ldots \oplus \;  hom(L,L_0)\otimes \overbrace{ (\e) \ldots (\e)}^{n-1} \otimes L
\eq
with differential acting on the $r^{th}$ summand by 
\bq
Id^{\otimes r-2} \otimes ev \; \oplus \; \sum_{r = i+j+k, \, j>1} Id^{\otimes i} \otimes \mu^j_\A \otimes Id^{\otimes k-1} \otimes 1.
\eq

\end{bartwistedcomplex}

This requires a weak form of Proposition \ref{th:lagspherestructure}: we only need the structure on the cohomology level. 
The twisted complex \eqref{eq:twistedcomplexintro} immediately gives an expression for $hom(L_1, \tau^n_L L_0)$. 
Consider the $A_\infty$ reduced bar complex corresponding to the tensor over $A$ of $ hom(L, L_0) =: M$ and $hom(L_1, L)=:N$, say 
\begin{multline}\label{eq:barcomplex}
M \otimes_A N := \\
\big(M \otimes N \big) \oplus  \big(M \otimes (\e) \otimes N \big)
\oplus \big( M \otimes (\e)  \otimes (\e) \otimes N \big)
\oplus \ldots \oplus 
\big( M \otimes (\e)  \otimes \ldots \otimes (\e) \otimes N \big)
 \oplus  \ldots
\end{multline}
where tensor products without subscripts are taken over $\Z_2$, and the differential is given by summing over all the possible $Id^{\otimes i} \otimes \mu^j_{\A} \otimes Id^{\otimes k}$. (See \cite[Section 2.3.3]{Hasegawa}.) We shall denote by $(M \otimes_A N)_n$ the truncation of complex \eqref{eq:barcomplex} that consists of  only the first $n$ summands.
We find that
\bq
hom(L_1, \tau^n_L L_0) = hom(L_1, L_0) \oplus (hom(L, L_0) \otimes_A hom(L_1, L) )_n
\eq 
with the obvious differential, again obtained by taking all possible $A_\infty$ products. Moreover, we get the inequality
\begin{equation}\label{eq:truncatedbarcomplexinequality}
dim (HF( \tau_L^n L_0, L_1)) + dim(HF(L_0, L_1)) \geq dim \big( (hom(L, L_0) \otimes_A hom(L_1, L))_n \big).
\end{equation}
This is used to prove weak analogs of the inequality used by Ishida, by estimating the right-hand side. In order to do so, we first classify finite dimensional $A_\infty$ left and right modules over $A$ (Proposition \ref{th:classificationofmodules}), of which $hom(L, L_0)$ and $hom(L_1, L)$ are examples. This implies the following:
\begin{rankmasterinequalities}
For all integers $n \neq 0$, we have
\bq\label{eq:ineq1}
dim(HF(\tau^n_L (L_0), L_1)) + dim(HF(L_0, L_1)) \geq dim(HF (L, L_1)) \cdot dim(HF (L_0, L)).
\eq
Further, if $L \ncong L_0, L_1$ in the Fukaya category, and $|n| \geq 2$, we find that
\bq\label{eq:ineq2}
dim(HF(\tau^n_L (L_0), L_1)) + dim(HF(L_0, L_1)) \geq 2\,dim(HF (L, L_1)) \cdot dim(HF (L_0, L)).
\eq

\end{rankmasterinequalities}
Inequality \eqref{eq:ineq2} also uses Corollary \ref{th:multiplybyenontrivialimage}. 
Inequalities \eqref{eq:ineq1} and \eqref{eq:ineq2} are strong enough to allow us to conclude the proofs of Theorems \ref{th:freegroup} and \ref{th:freegrouponFukaya}, in section \ref{sec:conclusion}, by following Ishida's argument \cite{Ishida}. The case $dim(HF(L,L'))=2$ requires careful consideration, and additionally, the use of Corollary \ref{th:surjectivegiveiso}.

Section \ref{sec:examples} contains the examples. 
We consider Milnor fibres of isolated hyperplane singularities with the natural symplectic form. 
Classically, if a singularity is adjacent to another, one gets a smooth embedding of one Milnor fibre into another; there is an analogous fact in the symplectic setting. By a result of Abouzaid (Lemma \ref{th:Abouzaid}), the structure of the Floer complex between two Lagrangians is preserved under such an embedding. This implies that if a singularity $[f]$ is adjacent to $[g]$, and the Milnor fibre of $g$ contains Lagrangian spheres satisfying the hypothesis of Theorem \ref{th:freegroup}, then so do the Milnor fibres of representatives of $[f]$ (Corollary \ref{th:propertySadjacency}). It is thus enough to construct strategically positioned such spheres, in the Milnor fibre of the $(A_2)$ singularity (and, to get the homologous odd-dimensional examples, $(A_3)$). We do this by hand, using the framework of Khovanov and Seidel \cite{KhovanovSeidel}. 
In the even dimensional case, we shall get `for free' that the two spheres are smoothly isotopic. In the odd dimensional case, we show that one of the spheres that we construct in the $(A_3)$ fibre can be isotoped so as not to intersect the other using the Whitney trick.

\subsubsection*{Acknowledgements}
I thank my advisor, Paul Seidel, for suggesting this project, and for many helpful conversations and suggestions. 
I would also like to thank Mohammed Abouzaid for several patient explanations, notably of material presented in section \ref{sec:Fukayapreliminaries}, and Nicholas Sheridan for comments on an earlier version of this draft. Additionally, I thank Oscar Randal-Williams for pointing out reference \cite{Krylov} to me.

 Part of this work was carried out during visits to the Simons Center for Geometry and Physics, which I thank for its hospitality. 
I was partially supported by NSF grant DMS-1005288. 


\section{$A_\infty$ categories and twisted complexes thereof: preliminaries}\label{sec:Ainfinitypreliminaries}


\subsection{Definitions and notation}

Our $A_\infty$ categories are defined over the field of two elements $\Z_2$, and do not have gradings.
An $A_\infty$ category $\A$ will consist of a set of objects $\text{Ob} \A$, and, for each ordered pair of objects $(L_0, L_1)$,  a $\Z_2$--vector space of morphisms between them, denoted by $hom_\A (L_0, L_1)$.
This is equipped with additional structure, the $A_\infty$ composition maps. We follow \cite{Seidel08}'s convention on the ordering of indices for these:
\bq
\mu^d_\A : hom_\A(L_{d-1}, L_d) \otimes \ldots \otimes hom_\A(L_0, L_1) \to hom_\A (L_0, L_1).
\eq
These must satisfy relations, usally called the $A_\infty$ associativity equations
\bq
\sum_{n=r+s+t,\, s \geq 1} \mu_\A^{r+t+1} (1^{\otimes r} \otimes \mu_\A^s \otimes 1^{\otimes t} ) =0
\eq
for each $n$. 
When there is no ambiguity about which category we are working in, we will suppress the subscript $\A$ on both $hom$ and $\mu^d$. 
For a given $\A$, the cohomology category $H(\A)$ has the same objects, and morphisms $Hom(L_0, L_1):=H(hom(L_0, L_1), \mu^1)$.
This inherits a product from $\mu^2$, defined by
\begin{eqnarray}
Hom(L_1, L_2) \otimes Hom(L_0, L_1) &  \to & Hom(L_0, L_2) \\
([b] , [a]) & \mapsto & b \cdot a := [\mu^2(b,a)]
\end{eqnarray}
for any two cocycles $a \in hom(L_0, L_1)$, $b \in hom(L_1, L_2)$. 
By a `representative' of an $A_\infty$ category we will mean any quasi-isomorphic category. 
\newline

\noindent
An $A_\infty$ category might additional have one, or more, of the following features:
\begin{itemize}
\item finiteness: the morphism spaces are finite-dimensional;
\item cohomological finiteness: the morphism spaces of $H(\A)$ are finite-dimensional;
\item minimality: $\mu^1$ vanishes;
\item c-unitality: the cohomology category $H(\A)$ is unital;
\item strictly unitality: for each $L \in \text{Ob} \A$, there is an element $1_L \in hom(L,L)$ such that $\mu^1(1_L) = 0$, $\mu^2(a, 1_L) = \mu^2(1_L', a) = a$ for all $a \in hom(L, L')$, and higher compositions involving $1_L$ are zero.
\end{itemize}


\subsection{Choosing `nice' representatives for $A_\infty$ categories}

Given any $A_\infty$ category $\A$, we can always construct a minimal quasi-isomorphic $A_\infty$ category with the same objects.
If $\A$ is also c-unital, minimality and strict unitality can be achieved simultaneously (\cite[Lemma 2.1]{Seidel08}). 
Also, if two strictly unital $A_\infty$ algebras are quasi-isomorphic, and the quasi-isomorphism is c-unital, then the algebras can be related by a quasi-isomorphism that is itself strictly unital (\cite[Theorem 3.2.2.1]{Hasegawa}). 

Fix any minimal, strictly unital $A_\infty$ algebra $B$. 

\begin{lem}\label{th:nicerepresentative}
Suppose $\A$ is a c-unital $A_\infty$ category with a distinguished object $L$ such that the $A_\infty$ algebra $hom_\A(L,L)$ is quasi-isomorphic to $B$, and that the quasi-isomorphism is c-unital. Then we can find a quasi-isomorphic category $\tilde{\A}$, with the same objects, that is minimal, strictly unital, and such that $hom_{\tilde{\A}}(L, L)$ is strictly isomorphic to $B$.
\end{lem}

\begin{proof}
First fix an $A_\infty$ category $\A'$ that is minimal, strictly unital, and quasi-isomorphic to $\A$, with the same objects. The $A_\infty$ algebra $hom_{\A'}(L,L)$ is strictly unital, and quasi-isomorphic to the strictly unital algebra $B$. We can find a quasi-isomorphism between $hom_{\A'}(L,L)$ and $B$ that is strictly unital. Say it is given by maps
\bq f_n: hom_{\A'}(L,L)^{\otimes n} \to B.
\eq
As $\A'$ and $B$ are minimal, $f_1$ is just an automorphism, and the $f_n$ can be extended to a formal diffeomorphism (as defined in \cite[Section 1]{Seidel08}) on $\A'$. Moreover, we can require that the maps 
\bq hom_{\A'}(X_{d-1}, X_d) \otimes \ldots \otimes hom_{\A'}(X_0,X_1) \to hom_{\A'}(X_0, X_1)
\eq
giving the formal diffeomorphism vanish whenever $d \geq 2$, $X_i = X_{i+1}$ for some $i$, and the element of $hom_{\A'}(X_i, X_{i+1})$ we are plugging in is the strict unit. Call $\tilde{\A}$ the $A_\infty$ category produced by the formal diffeomorphism; it is easy to check that our requirement forces $\tilde{\A}$ to be strictly unital, and that the units are the images of the units of $\A'$. Thus $\tilde{A}$ is a suitable category.
\end{proof}

\begin{remark}
If $\A$ is finite or cohomologically finite, and (where applicable) $B$ is finite dimensional, then it is immediate that the quasi-isomorphic categories described above are also finite.
\end{remark}


\subsection{Background on twisted complexes}

Let $\A$ be a finite, strictly unital $A_\infty$ category. We summarize the material we will use on twisted complexes in $\A$, as introduced by \cite{Kontsevich}. If $\A$ were not unital, one would need far more care; although most computations are routine and omitted, we have tried to flag all times that strict unitality is used.

\paragraph{Additive enlargement.}

The additive enlargement of $\A$, $\Sigma \A$, is an $A_\infty$ category whose objects are formal direct sums
\bq
X = \bigoplus_{i \in I} V^i \otimes X^i
\eq
where $I$ is a finite set, $\{X^i\}_{i\in I}$ a family of objects of $\A$, and $\{V^i\}_{i\in I}$ a family of finite dimensional $\Z_2$ vector spaces.
The morphisms between two of these objects are made up of morphisms between the constituant summands, tensored with the spaces of linear maps between the vector spaces: for instance,
\bq hom_{\Sigma\A} (V^0 \otimes X^0, V^1 \otimes X^1) = hom_{\Z_2}(V^0,V^1) \otimes hom_\A (X^0, X^1).
\eq
 The $A_\infty$ compositions maps are inherited from those of $\A$, combined with usual composition of linear maps. We will often denote both vector and $A_\infty$  morphism spaces by $hom$. The identity endomorphism of a vector space $V$ will be denoted $Id_V$ or $Id$, whereas the strict unit of $hom(L,L)$ will be denoted by $1_L$ or $1$.

\paragraph{Twisted complexes.}

We will work with the category of twisted complexes in $\A$, $Tw\A$.
Objects consist of pairs $(X, \delta_X)$, where $X \in \text{Ob}\Sigma\A$, and the connection (or differential) $\delta_X$ is an element of $hom_{\Sigma\A}(X,X)$ such that
\begin{itemize}
\item there is a finite, decreasing filtration by subcomplexes 
\bq X = F^0 X \supset F^1X \supset \ldots \supset F^nX =0
\eq
such that the induced connection on the quotients $F^i X / F^{i+1}X$ is zero. (Subcomplexes consist of objects $\bigoplus_{i\in I}W^i \otimes X^i$, where each $W^i$ is a vector subspace of $V^i$, that are preserved by the connection $\delta_X$.)
\item the connection satisfies the generalised Maurer-Cartan equation
\bq \sum_{r=1}^\infty \mu^r_{\Sigma\A} (\delta_X, \ldots, \delta_X) = 0.
\eq
\end{itemize}
The morphism spaces are the same as for $\Sigma \A$; hereafter the subscript will be dropped. A key difference with $\Sigma\A$ is that all the compositions are now deformed by contributions from the connections:
\bq 
\mu^d_{Tw\A} (a_d, \ldots, a_1)= \sum_{i_0, \ldots, i_d} \mu^{d + i_0 + \ldots + i_d}_{\Sigma\A} (\overbrace{\delta_{X_d}, \ldots, \delta_{X_d}}^{i_d}, a_d , 
\overbrace{\delta_{X_{d-1}}, \ldots, \delta_{X_{d-1}}}^{i_{d-1}}, a_{d-1},
 \ldots, a_1, 
\overbrace{\delta_{X_0}, \ldots, \delta_{X_0}}^{i_0})
\eq
where the sum is over all non-negative integers $i_k$. This makes $Tw\A$ into an $A_\infty$ category. We will denote the associated cohomology category by $H(Tw\A)$, and its morphism groups by $Hom_{Tw\A}$. 
$\A$ embeds into $Tw\A$ as a subcategory: each object $X \in A$ gives an object $(X, \delta_X = 0)$ of $Tw\A$.
If $\A$ is finite, $Tw\A$ clearly is too. Moreover, strict unitality of $\A$ implies strict unitality of $Tw\A$. 
$H(Tw\A)$ is a triangulated category (see \cite{GelfandManin}). 
notice that the shift functor $[1]$ is the identity.
Notably, this means there is a collection of distinguished triangles; for each such triangle, say 
\begin{equation} \xymatrix@C=2em{
 {A} \ar[rr]_{f} & & {B} \ar[dl]_{g} \\
 & {C} \ar[ul]_{h}  & }
\end{equation}
 and any $Z \in \text{Ob}Tw\A$, there are exact sequences
\begin{equation} 
\ldots \to Hom_{Tw\A}(Z,A) \to  Hom_{Tw\A}(Z, B)  \to Hom_{Tw\A}(Z, C) \to \ldots
\end{equation}
and
\begin{equation}
 \ldots \to {Hom_{Tw\A}(B, Z)} \to {Hom_{Tw\A}(A,Z)} \to {Hom_{Tw\A}(C, Z)} \to \ldots
\end{equation}
 in the cohomology category, 
where the maps are pre- and post-composition with $f$, $g$ and $h$.

\paragraph{Quotients and cones.}
Given a twisted complex $(X, \delta_X)$ and a subcomplex $(Y, \delta_X|Y)$, one can form the quotient complex by taking vector space quotients piece-wise; this inherits a connection from $X$. We denote this by $(X/Y, \delta_{X/Y})$.
Let $\pi: X \to X/Y$  be the quotient map, and $i: Y \to X$ the inclusion. These are cochains (this uses strict unitality), so make sense as a morphisms in the cohomology category $H(Tw\A)$. Moreover, there is a distinguished triangle 
\begin{equation} \xymatrix@C=2em{
 {X/Y} \ar[rr] & & {Y} \ar[dl]_{i} \\
 & {X} \ar[ul]_{\pi}  & }
\end{equation}
Consider a morphism between twisted complexes, say 
\bq c: X \to Y.
\eq
 Whenever $\mu^1_{Tw\A} (c) = 0 \in hom (X,Y)$, we can complete this to a distinguished triangle as follows: define the cone of $c$, itself an element of $Tw\A$:
\bq
Cone(c) = X \oplus Y
\eq
with 
\begin{equation}
\delta_{Cone(c)} = 
\left(
\begin{array}{ccc}
\delta_{X} & 0 \\
c & \delta_{Y}  
\end{array} \right)
\end{equation}
As we are working with $\Z_2$ coefficients and no gradings, there is not need for the shifts or negative signs that the reader might be used to. 
Together with $c$, the obvious maps $Cone(c) \to X$ (projection) and $Y \to Cone(c)$ (inclusion) fit into a distinguished triangle.

\begin{lem}\label{th:isobetweentwosummands}
Suppose $(X, \delta_X)$ is a twisted complex, $(Y, \delta_X|Y)$ a subcomplex, and that $Y$ sits in a distinguished triangle with $A \xrightarrow{f} B$, such that $f$ is an isomorphism. Then $Y$ is acyclic, and $\pi : X \to X/Y$ is an isomorphism in $H(Tw(\A))$. 
\end{lem}

\begin{proof}
Consider the distinguished triangle
\begin{equation} \xymatrix@C=2em{
 {A} \ar[rr]_{f} & & {B} \ar[dl] \\
 & {Y} \ar[ul]  & }
\end{equation}
For each $Z \in \text{Ob}Tw(\A)$ there is an exact sequence
\begin{equation} \ldots \to
 {Hom_{Tw\A}(Z,A)} \xrightarrow f_{\ast} {Hom_{Tw\A}(Z, B)} \to
 {Hom_{Tw\A}(Z, Y)} \to \ldots
\end{equation}
and by hypothesis, post-composition $f_{\ast}$ is an isomorphism; thus, by Yoneda's lemma, $Y$ is isomorphic to 0 in $H(Tw(\A))$. Considering the distinguished triangle
\begin{equation} \xymatrix@C=2em{
 {X/Y} \ar[rr] & & {Y} \ar[dl] \\
 & {X} \ar[ul]  & }
\end{equation}
another use of Yoneda's lemma gives the desired conclusion.
\end{proof}

\paragraph{Evaluation maps.}

For any pair $X, Y \in \text{Ob }Tw\A$, as $hom(X, Y)$ is finite, $hom(X,Y) \otimes X$ is an element of $\Sigma\A$; moreover, setting 
$$\delta_{hom(X,Y) \otimes X} = \mu_{Tw\A}^1\otimes 1+ Id \otimes \delta_X
$$
gives it the structure of a twisted complex. (This uses strict unitality; see e.g.~\cite{Seidel08}, section (3o).) 
There are canonical isomorphisms
\begin{multline} hom (hom(X,Y)\otimes X, Y) \cong hom_{\Z_2}(hom(X,Y), \Z_2) \otimes hom(X,Y)
\\ \cong hom(X,Y)^\text{v} \otimes hom(X,Y) 
\cong End_{\Z_2} \big( hom(X,Y) \big).
\end{multline}
The evaluation map 
\bq
ev: hom(X,Y)\otimes X \longrightarrow Y
\eq
corresponds the the identity in $End_{\Z_2} \big( hom(X,Y) \big)$: for any $\Z_2$-vector space basis for $hom(X,Y)$, say $e_i$, we have
\bq ev = \sum e_i^\text{v} \otimes e_i
\eq
One can the check that $\mu^1_{Tw\A}(ev)=0$. 

(This too uses strict unitality of $Tw\A$.) Thus we can talk about its cone, the twisted complex $Cone(ev)$. 
We shall later use the following observation:
\begin{lem}\label{th:evmaps}
Let $X_i$ be objects of $\Sigma\A$, for $i=0, \ldots, r$. We denote by $ev_i$ the evaluation map
\bq hom(X_i,X_{i+1}) \otimes X_i \to X_{i+1}
\eq
and let 
\bq Z_i := hom(X_{r-1},X_r) \otimes \ldots \otimes hom(X_i,X_{i+1}) \otimes X_i.
\eq
Tensoring with the identity, $ev_i$ extends to a map $Z_i \to Z_{i+1}$, that we also denote by $ev_i$. Fix $W \in \text{Ob}\Sigma\A$, and let $a_r \otimes \ldots \otimes a_{i+1} \otimes b$ denote an element of
\bq hom(W, Z_i) = hom(X_{r-1}, X_r) \otimes \ldots \otimes hom(X_i, X_{i+1}) 
\otimes hom(W, X_i).
\eq
Then
\bq
\mu^{r-i+1}_{\Sigma\A} (ev_r, ev_{r-1}, \ldots,  ev_i , a_r\otimes a_{r-1} \otimes \ldots \otimes a_{i+1} \otimes b) 
=
\mu^{r-i+1}_{\Sigma\A} (a_r, a_{r-1}, \ldots, a_{i+1}, b) \in hom(W, X_r).
\eq
\end{lem}

\begin{proof}
This is nothing more than an exercise in definitions. For instance, in the simplest case, we have $a \in hom(X_0, X_1)$, $b \in hom(W, X_0)$, and 
\bq
ev: hom(X_0, X_1) \otimes X_0 \to X_1.
\eq
Write $ev = \sum e_i^\text{v} \otimes e_i \in hom(X_0, X_1)^\text{v} \otimes hom(X_0, X_1)$, for some basis $e_i$ of $hom(X_0, X_1)$, and $a = \sum a_i \cdot e_i$, some scalars $a_i$. We have
\begin{eqnarray}
\mu^2_{\Sigma \A} (ev, a \otimes b) & = & \mu^2_{\Sigma \A} (\sum e_i^\text{v} \otimes e_i, a \otimes b) \\
& = &  \sum \mu^2_{\Sigma \A} ( e_i^\text{v} \otimes e_i, a \otimes b) \\
& = &  \sum ( e_i^\text{v} \circ a ) \mu^2_\A (e_i , b) \\
& = & \sum a_i \mu^2_{\A} (e_i, b) \\
& = & \mu^2_{\A} (\sum a_i \cdot e_i, b) \\
& = & \mu^2_{\A} (a, b).
\end{eqnarray}
\end{proof}


\section{The Fukaya category of an exact manifold: preliminaries}\label{sec:Fukayapreliminaries}

Let $(M^{2n},\omega, \theta)$ be an exact symplectic manifold of dimension $2n >2$ with contact type boundary, and $J$ an $\omega$-compatible almost complex structure of contact type near the boundary. 
Unless otherwise specified, all Lagrangians will hereafter be assumed to be exact, compact, and disjoint from the boundary. Note Lagrangian spheres are automatically exact.

We will use the Fukaya category $\Fuk(M)$ of $M$, as introduced in \cite{Fukaya93}; we follow the exposition of \cite[Section 9]{Seidel08}. 
 $\Fuk(M)$ is an $A_\infty$ category with objects the Lagrangians of $M$, and morphisms Floer chain groups between Lagrangians. Our Floer chain complexes have $\Z_2$ coefficients, and there are no gradings.
We start by giving a rapid overview of the set-up, calling attention to some of the features we shall use (subsections \ref{sec:Floercohomology} and \ref{sec:Fukayacategory}). We then describe the $A_\infty$ algebra structure of $CF(L,L)$, for a Lagrangian sphere (subsection \ref{sec:spherestructure}). Finally, subsection \ref{sec:isomorphismcriteria} gives some criteria for two Lagrangian spheres to be quasi-isomorphic as objects of the Fukaya category.

\subsection{Floer cohomology}\label{sec:Floercohomology}

Fix Lagrangians $L_0$ and $L_1$.
Let $\H = C^{\infty}_c (\text{Int}(M), \R)$, and let  $\J$ be the space of all $\omega$-compatible almost complex structures that agree with $J$  near the boundary of $M$.

 A \emph{Floer datum} for the pair $(L_0, L_1)$ consists of $ J_{L_0, L_1} \in C^{\infty}([0,1], \J)$ and $H_{L_0, L_1} \in C^{\infty}([0,1], \H)$ 
such that $\phi^1(L_0) \pitchfork L_1$, where $\phi$ is the flow of the time-dependent Hamiltonian vector field $X$ corresponding to $H$. Let 
\bq
\CC(L_0, L_1) = \{ y: [0, 1] \to M \, | \, y(0) \in L_0, \,y(1) \in L_1, \,dy/dt = X(t, y(t)) \}.
\eq  Its elements are naturally in bijection with the set $\phi^1(L_0) \pitchfork L_1$. 

Given a Floer datum, the Floer cochain group $CF(L_0, L_1)$ is the $\Z_2$ vector space generated by the elements of $\CC(L_0, L_1)$. One equips it with a differential, $\mu^1$, that counts certain `pseudo--holomorphic strips' as follows: fix $y_0$, $y_1 \in \CC(L_0, L_1)$. Let $Z$ be the Riemann surface $\R \times [0,1] \subset \C$, with coordinates $(s,t)$. $\M_Z(y_0, y_1)$ is the set of maps $u \in C^\infty(Z, M)$ satisfying:
\begin{itemize}
\item Floer's equation: $\partial_s u + J(t,u(s,t))(\partial_t u - X(t,u)) =0$ 
\item boundary conditions: $u(s,0) \in L_0$, $u(s,1) \in L_1$ 
\item asymptotic conditions: $\text{lim}_{s \to + \infty} u(s, \cdot) = y_1$, $\text{lim}_{s\to -\infty} u(s, \cdot) = y_0.$ 
\end{itemize}
There is an $\R$-action on $\M_Z(y_0, y_1)$: translation in the $s$-variable. For $y_0 \neq y_1$, this action is free; let $\M^\ast_Z(y_0, y_1)$ be the quotient space. For $y_0 = y_1$, $\M_Z(y_0, y_1)$ is a point, because of exactness, and we set $\M^\ast_Z (y_0, y_1) = \emptyset$.

Additionally, we put a constraint on our Floer datum: we require that the space $\M_Z(y_0, y_1)$ defined above be \emph{regular}; this is a property that implies that it is smooth and of the appropriate, expected dimension, given by the index of a certain surjective Fredholm operator. This condition is called transversality; see e.g. \cite{FloerHoferSalamon, Oh}. It is satisfied for a generic choice of Floer datum. Note that we are free to choose any $H_{L_0, L_1}$ such that $\phi^1(L_0) \pitchfork L_1$, as long as we don't want to exerce any additional control on $J_{L_0, L_1}$. 

Given a regular Floer datum, the boundary operator on $CF(L_0, L_1)$, called the Floer differential, is given by
\bq
\mu^1(y_1) = \sum_{y_0} \# \M^\ast_Z (y_0, y_1) \cdot y_0
\eq
where $\#$ counts the number of isolated points mod 2.

\paragraph{Properties.}

The boundary operator $\mu^1$ is a differential. The associated cohomology, 
\bq
H(CF(L_0, L_1), \mu^1) =:HF(L_0, L_1)
\eq
is called the Lagrangian Floer cohomology of $L_0$ and $L_1$. 

$HF(L_0, L_1)$  is independent of auxiliary choices, up to canonical isomorphism. It is independent of the choice of $\theta$ away from $\partial M$: we could have chosen instead $\theta + df$, for any $f \in C^\infty(M, \R)$ supported away from $\partial M$. Also, it is finite as a $\Z_2$ vector space. It is invariant under exact Lagrangian isotopy of $L_0$ or $L_1$. 
\begin{rmk}
Whenever the Floer cochain group can be graded by Maslov indices, as in e.g. \cite{RobbinSalamon}, $\mu^1$ increases the degrees by one. This is why we talk of Floer \emph{co}chains and \emph{co}homology.
\end{rmk}
For an two Lagrangians $L_0$ and $L_1$, we will denote by $hf(L_0,L_1)$ the dimension of $HF(L_0,L_1)$ as a $\Z_2$ vector space.

\paragraph{Duality.}

Suppose $(H_{L_0,L_1}, J_{L_0,L_1})$ is a regular Floer datum for the pair $(L_0, L_1)$. Then a regular Floer datum for $(L_1, L_0)$ is given by
\bq
H_{L_1, L_0} (t) = - H_{L_0, L_1} (1-t) \qquad J_{L_1, L_0}(t) =  J_{L_0, L_1}(1-t).
\eq
We will refer to this as the `dual Floer datum' for $(L_1, L_0)$. 
Given these choices, there is a one-to-one correspondance between generators of $CF(L_0, L_1)$ and $CF(L_1, L_0)$, and between the Floer discs counted by their differentials $\mu^1$. These identify $CF(L_1,L_0)$ with the dual complex of $CF(L_0, L_1)$; we will call corresponding generators `dual' to one another.  One gets a vector space isomorphism $HF(L_0, L_1) \cong (HF(L_1, L_0))^\text{v}$. 


\subsection{The Fukaya category}\label{sec:Fukayacategory}

Suppose you have already chosen a regular Floer datum $(H_{L_0, L_1}, J_{L_0, L_1})$ for each pair of Lagrangians $(L_0, L_1)$. The Fukaya category has objects the Lagrangians of $M$, and morphisms Floer chain groups between Lagrangians. We want to define maps
\bq
\mu^d: \, CF(L_{d-1}, L_{d}) \otimes CF(L_{d-2}, L_{d-1}) \otimes \ldots \otimes 
CF(L_0, L_1) \to CF(L_0, L_d)
\eq
for $d \geq 2$, that, together with the Floer differential, should satisfy the $A_\infty$ associativity equations for each $n$. 
Roughly speaking, the $\mu^d$ are obtained by counting certain `pseudo-holomorphic' discs with $d+1$ boundary marked points. Here are a few more details. 

Fix $d$. Let $\D$ be a closed unit complex disc with some $d+1$ boundary points removed; label these $\zeta_0, \ldots, \zeta_d$, ordered anti-clockwise. $\zeta_0$ will be known as an incoming point, and $\zeta_i$ ($i \geq 1$) as outgoing ones. Consider the half-infinite holomorphic strips $\R^{\pm} \times [0,1] \subset \C$. 
Equip $\D$ with \emph{strip-like ends}: proper holomorphic embeddings
\begin{center}
$\e _{0}: \R^- \times [0,1] \to \D$ \, and \,
$\e_{j}: \R^+ \times  [0,1] \to \D$ for $j = 1, \ldots, d$
\end{center}
with disjoint images,  such that 
$\e_j^{-1}(\partial \D) = \R^{\pm} \times \{0,1\}$ and $\text{lim}_{s \to \pm \infty} \e_j (s, \cdot) = \zeta_j$ for $j = 0, \ldots, d$. 

  The boundary $\partial \D$ consists of $d+1$ connected components; let $C_i$ be the segment between $\zeta_i$ and $\zeta_{i+1}$, with indices taken modulo $d+1$. \emph{Lagrangian labels} for $D$ are the asignment of a Lagrangian, say $L_i$, to each $C_i$.

A \emph{perturbation datum} for $\D$ is a pair $(K, J)$ where $J \in C^\infty(\D, \J)$ and $K \in \Omega^1(\D, \H)$ such that $K(\xi)|L_j = 0$ for all $\xi \in TC_j \subset T(\partial S)$, any $j$. 
We require it to be compatible with choice of strip-like ends and the Floer data, in the following sense:
\bq
\e_{\zeta_j}^\ast K = H_{\zeta_j} (t) dt, \quad J(\e_{\zeta_j}(s,t)) = J_{\zeta_j}(t)
\eq
for all $j$, and all $(s,t) \in \R^{\pm} \times [0,1]$. 

Let $Y \in \Omega^1 (\D, C^\infty (TM))$ be the Hamiltonian vector field valued one-form associated to $K$. Given $a_i \in CF(L_{i-1}, L_1)$, the naive approach would be to define $\mu^d(a_d, \ldots, a_1)$ by counting solutions $u \in C^\infty( \D , M) $ to the generalized Floer equation
\bq
Du(z) + J(z,u) \circ Du(z) \circ i = Y(z,u) + J(z,u) \circ Y(z,u) \circ i
\eq
with boundary conditions given by the Lagrangian labels $L_i$ and asymptotic conditions given by the $a_i $, and to sum over all possible $\D$ (allowing the marked points to move). As with Floer's equation, one would additionally need the perturbation data to be regular.

The problem is that for the $A_\infty$ relations to be satisfied, the auxiliary data needs to be chosen far more carefully. There are various solutions to this, one of which (\cite{Seidel08}) uses strip-like ends and perturbation data defined on classifying spaces for families of holomorphic discs with $d+1$ marked points, called `universal' choices. It is shown that these choices can be made to be `consistent': roughly, this enables one to carry through the requisite gluing arguments. Analogously to the discussion before, the consistent universal choice of perturbation data is required to be compatible with the Floer data and universal choice of strip-like ends.
These choices are generically regular.

The construction of such strip-like ends and perturbation data starts with any choice of regular Floer data, and inducts on $d$. We note the following feature of this induction: for a fixed $d$, and suppose we are looking to equip discs with $d+1$ marked points with adequate perturbation data. For a given collection of Lagrangian labels, the choice of perturbation datum made in the induction only depends on the data for discs with $k+1$ marked points, $k < d$, and Lagrangian labels an ordered subset of our collection.

\paragraph{Properties.}

This gives an $A_\infty$ category. Different admissible choices of auxiliary data (strip-like ends, Floer and perturbation data) give a quasi-isomorphic $A_\infty$ category; the quasi-isomorphism fixes the objects. Working instead with $(M, \omega, \theta +df)$, for some $f \in C^{\infty}(M, \R)$ with support away from $\partial M$, gives an isomorphic $A_\infty$ category.
Thus the product that the cohomology category $H(\Fuk(M))$ inherits from $\mu^2$ is independent of all these choices. Moreover, with this structure $H(\Fuk(M))$ is  unital (it is a linear category), which means that $\Fuk(M)$ is c-unital.  In particular, it makes sense to talk about quasi-isomorphic objects (see section \ref{sec:isomorphismcriteria}); notice that any two Lagrangians that are isotopic through exact Lagrangians are quasi-isomorphic as objects of $\Fuk(M)$. 


\subsection{$A_\infty$ algebra associated to a Lagrangian sphere.}\label{sec:spherestructure}

\paragraph{Zero-section of a cotangent bundle.}

Let $L$ be a sphere of dimension at least 2. $T^\ast L$ has a standard exact symplectic structure; denote the usual one-form by $\alpha$, and by $\omega_L = d \alpha$ the symplectic form. The zero-section is an exact Lagrangian; we will simply denote it by $L$. Fix a metric on $L$; for any $\delta>0$, the closed disc bundle of radius $\delta$, say $B_\delta$, is an exact sympectic manifold with contact type boundary. We can find an $\omega$-compatible almost complex structure $j$, of contact type near the boundary.

\begin{lemma}\label{th:zerosectioncotangent}
Let $A$ be the $A_\infty$ algebra on two generators given by $\Z_2[\e]/(\e^2)$, with $\mu^2(1,\e) = \mu^2(\e,1) = \e$, and all other $A_\infty$ structure maps trivial. $CF_{B_\d}(L,L)$ is quasi-isomorphic to $A$.
\end{lemma}

Many variations or partial forms of this statement exist in the literature, going back to \cite{FukayaOh}. One might think of this fact as a consequence of two features of Floer cohomology:

\begin{itemize}
\item {\bf PSS isomorphism.} $HF(L,L)$ is isomorphic, as a ring, to the usual (e.g. simplicial or singular) cohomology ring $H^\ast (L; \Z_2)$. The maps between them are usually called `PSS isomorphisms'. For Lagrangian Floer theory, these were constructed by Albers \cite{Albers}, following work of Piunikhin, Salamon and Schwarz \cite{PSS} for Hamiltonian Floer theory. Albers does not discuss the ring structure, but one can check that \cite{PSS}'s arguments carry over -- see \cite{KMS}.
\item {\bf Gradings.} $c_1(T^\ast L) = 0$, and the zero section $L$ has Maslov class zero. $CF(L,L)$ carries a canonical $\Z$ grading (\cite[Theorem 2.3]{Fukaya97}), giving it the structure of a \emph{graded} $A_\infty$ algebra. 
\end{itemize}
To prove Lemma \ref{th:zerosectioncotangent}, it would be enough to check that the PSS isomorphisms are compatible with these gradings. Why? $CF(L,L)$ is certainly isomorphic to $A$ as a differential algebra, by PSS. Take a minimal, strictly unital model for $CF(L,L)$; it has a distinguished pair of generators, which, abusing notation slightly, we denote by $1$ and $\e$.  Compatibility with gradings would mean that 1 has grading 0, and $\e$ has grading $dim(L) \geq 2$. Together with strict unitality, these imply that all higher $A_\infty$ products must be zero.
\newline

Recent work of Abouzaid \cite{Abouzaid11} includes a detailed proof of Lemma \ref{th:zerosectioncotangent} (using a different strategy), so we do not discuss the above further.
\newline

Recall that when defining $CF(L,L)$, we are free to choose the Hamiltonian perturbation for the pair $(L,L)$. It will sometimes be useful to make the following choice: fix a Morse function on $L$ with two critical points, say $h$. 
Let $\psi: \R \to \R$ be a bump function centered at 0, with support on $[-\d/2, \d/2]$. Let $H$ be the function on $B_{\d}$ given by $H(q, p) = h(q) \psi(||p||)$, where $q \in L$, $p \in T^\ast_q L$, and $||\cdot||$ is our choice of metric. 
Set $H_{L,L}(t) = H$, for all $t \in [0,1]$, to be the Hamiltonian for $(L, L)$. 
Provided $h$ is sufficiently small, its critical points are by construction in one-to-one correspondence with generators of $CF(L, L)$. We assume this to be the case. Let $x_M$ be the critical point corresponding to the maximum, and $x_m$ the one corresponding to the minimum. One can check from the Morse cohomology gradings, and e.g. Abouzaid's proof, that on the level of cohomology, $x_M$ corresponds to 1, and $x_m$ to $\e$. (In our later discussion it will only actually matter that one critical point  can be identified with 1 and the other with $\e$.) We shall make use of this in the proof of Proposition \ref{th:multiplybye}.

\paragraph{Arbitrary exact Lagrangian sphere.} 
Let $L \subset M$ be a Lagrangian sphere. We claim that

\begin{prop}\label{th:lagspherestructure}
  $CF(L, L)$ is isomorphic to $A$ as an $A_\infty$ algebra.
\end{prop}

As before, let $B_\delta$ denote the disc bundle of radius $\delta$ of $T^\ast L$. 
By Weinstein, there exists $\delta>0$ and an embedding $\iota: B_\delta  \hookrightarrow M$ such that the zero section gets mapped to $L$, and $\iota^\ast \omega = \omega_L$. By exactness, $\iota^\ast \theta = \alpha + dg$, some smooth function $g$. 
Any Floer or (universal) perturbation datum for $B_\delta$ can be extended to one on $M$:
\begin{itemize}
\item Extend all Hamiltonian perturbations by zero.
\item $j$ induces an $\omega$-compatible almost complex structure on $\iota(B_\delta)$. As we are only using structures on $B_\delta$ that agree with $j$ on a neighbourhood of $\partial B_\delta$, it is enough to extend this. First extend $j$ to an $\omega_L$-compatible almost complex structure on $B_{2\delta}$. This gives an $\omega$-compatible almost complex structure on $\iota(B_{2\delta})$. Using the fact that the space of $\omega$-compatible almost complex structures on $\iota(B_{2\delta}\backslash B_{\delta})$ is contractible, we can find $j'$, an $\omega$-compatible almost complex structure on $M$, such that $j'$ restricts to $j$ on a neighbourhood of $\iota(B_\delta)$, and $j' = J$ outside $\iota(B_{2\delta})$.
\end{itemize}
Make the same choice of universal strip-like ends as for $B_\delta$. 

\begin{lem}{(Abouzaid, see e.g.~\cite{AbouzaidSeidel} or \cite[Lemma 7.5]{Seidel08})}\label{th:Abouzaid}
   Suppose $(N, \omega_N, \theta_N)$ is an arbitrary exact symplectic manifold, $(U, \omega_U, \theta_U)$ an exact symplectic manifold of the same dimension with contact type boundary, and $\iota: U \hookrightarrow$ int($N$) an embedding such that $\iota^\ast \omega_N = \omega_U$, $ \iota^\ast \theta_N = \theta_U + dg$. Suppose we also have an $\omega_N$-compatible almost complex structure $\mathbb{J}$, whose restriction to $U$ is of contact type near the boundary.

Let $S$ be a compact connected Riemann surface with boundary, and $u: S \to N$ a $\mathbb{J}$-holomorphic map such that $u(\partial S) \subset$ int($U$). Then $u(S) \subset$ int($U$) as well.
Moreover, this only requires $u$ to be $\mathbb{J}$-holomorphic in a neighbourhood of $\partial U$. (In particular, there could be a compactly supported perturbation on the interior of $U$.)

\end{lem}
This lemma implies that no pseudo-holomorphic disc corresponding to the extended data can leave $\iota(B_\delta)$.
Thus regularity of the data for $B_\delta$ implies regularity of its extension. 
The observation at the end of section \ref{sec:Fukayacategory} implies that this extended data can taken as part of consistent universal choices made to define $\Fuk(M)$. Suppose you have set-up the auxiliary data for $\Fuk(M)$ in this way. Using lemma \ref{th:Abouzaid}, we see that
\bq 
CF_{B_\delta}(L,L) = CF_M (L, L)
\eq
as $A_\infty$ algebras, `on the nose'. Proposition \ref{th:lagspherestructure} now follows from lemma \ref{th:zerosectioncotangent}.


\subsection{Isomorphism in the Fukaya category: some criteria}\label{sec:isomorphismcriteria}

$\Fuk(M)$ is c-unital, so there is a meaningful notion of quasi-isomorphism between its objects: two Lagrangians $L_0$, $L_1$ are quasi-isomorphic in $\Fuk(M)$ if there are $\mu^1$-closed morphisms $a \in CF(L_0,L_1)$, $b \in CF(L_1, L_0)$ such that
\bq [\mu^2(b,a)] = 1_{L_0} \in HF(L_0,L_0) \quad \text{and} \quad [\mu^2(a,b)] = 1_{L_1} \in HF(L_1,L_1) 
\eq
We will call two such objects ``Fukaya isomorphic" or ``quasi-isomorphic in the Fukaya category". When some of the objects are spheres, seemingly weaker conditions turn out to be equivalent to this one.

\begin{lem}\label{th:isojustneedonedirection}
Let $L_0$ be a Lagrangian sphere, and $L_1$ any Lagrangian. Suppose that we have $a \in CF(L_0,L_1)$, $b \in CF(L_1, L_0)$ such that 
\bq
[\mu^2(a,b)] = 1_{L_1} \in HF(L_1,L_1)
\eq
Then $L_0$ and $L_1$ are Fukaya isomorphic objects.
\end{lem}

\begin{proof}
$(b \cdot a )^2 = b \cdot a$, so $b \cdot a$ is an idempotent element of $HF(L_0, L_0)$. Moreover, as $a\cdot  b \cdot a = a$,$ b \cdot a$ is non-zero. As we know the ring structure of $HF(L_0, L_0)$, we can check that the only other idempotent is $1_{L_0}$ itself, which means that $b \cdot a = 1_{L_0}$.
\end{proof}


\begin{cor}\label{th:surjectivegiveiso}
Suppose that $L_0$ and $L_1$ are Lagrangian spheres, and that multiplication 
\bq
HF(L_0,L_1) \otimes HF(L_1, L_0) \to HF(L_1, L_1)
\eq
is surjective. Then $L_0$ and $L_1$ are Fukaya isomorphic. 
\end{cor}

\begin{proof}
It is enough to find elements $a \in HF(L_0,L_1)$ and $b \in HF(L_0,L_1)$ such that $a \cdot b = 1_{L_1} \in HF (L_1, L_1)$.
Consider the PSS ring isomorphism $HF(L_1, L_1) \cong H^\ast (L_1; \Z_2)$. We know that the image of multiplication $HF(L_0, L_1) \otimes HF(L_1, L_0) \to H^\ast (L_1; \Z_2)$ cannot be contained in $H^n(L_1)$. As all elements in $H^\ast(L_1; \Z_2) \backslash H^n(L_1; \Z_2)$ are invertible, there are $a \in HF(L_0, L_1)$ and $c \in HF(L_1, L_0)$ such that $a \cdot c$ is invertible in $HF(L_1,L_1)$. Let $d \in HF(L_1,L_1)$ be the inverse. Notice that by constrution, we have $a \cdot (c \cdot d) = 1_{L_1}$.
\end{proof}


\begin{prop}\label{th:multiplybye}
Suppose that $L_0$ is a Lagrangian sphere, and that the map
\begin{align}
HF (L_0, L_1) & \to HF (L_0, L_1)  \\
a & \mapsto a \cdot \e
\end{align}
is non-zero. Then there exists $d \in HF(L_1, L_0)$ such that $d \cdot a = 1_{L_0}$. 
\end{prop}

\begin{proof}
Fix regular Floer data for the pairs $(L_0, L_1)$ and $(L_0, L_0)$, say $(H_{L_0, L_1}, J_{L_0,L_1})$ and $(H_{L_0,L_0}, J_{L_0, L_0})$.
For convenience, assume $H_{L_0, L_0}$ is constructed from a Morse function on $L_0$, as described in section \ref{sec:spherestructure}. 
By hypothesis, there must be cocycles $a \in CF(L_0, L_1)$ and $c \in CF(L_0, L_1)$, non zero in cohomology, such that $\mu^2 (a, x_m) = c$. 
We can assume that $a$ corresponds to a unique point of $L_0 \pitchfork \phi^1(L_1)$, where $\phi^1$ is the time-one flow associated to $H_{L_0, L_1}$.

Let $e_i$ be generators of $CF_{\tilde{\F}}(L_0, L_1)$, each corresponding to a point of $L_0 \pitchfork \phi^1(L_1)$.
Up to holomorphic reparametrization, there is a unique closed disc $\D$ in $\C$ with three boundary points removed. W.l.o.g. it is the unit disc with the roots of unity removed. Label these anti-clockwise as $\zeta_0$ (incoming), $\zeta_1$ and $\zeta_2$ (both outgoing). 
Suppose you have chosen the auxiliary data to define a Fukaya category $\F$ of $M$, using $(H_{L_0, L_1}, J_{L_0,L_1})$ and $(H_{L_0,L_0}, J_{L_0, L_0})$. $\D$ has been equipped with strip-like ends, and a perturbation datum for each choice of Lagrangian labels. The coefficient of $e_i$ in the product $\mu^2(a, x_m)$ is computed by counting certain solutions $u \in C^\infty(\D, M)$ to the generalized Floer equation corresponding to the perturbation datum, and boundary and asymptotic conditions determined by figure \ref{fig:holodisc1}. 

\begin{figure}[htb]
\begin{center}
\includegraphics[scale=0.85]{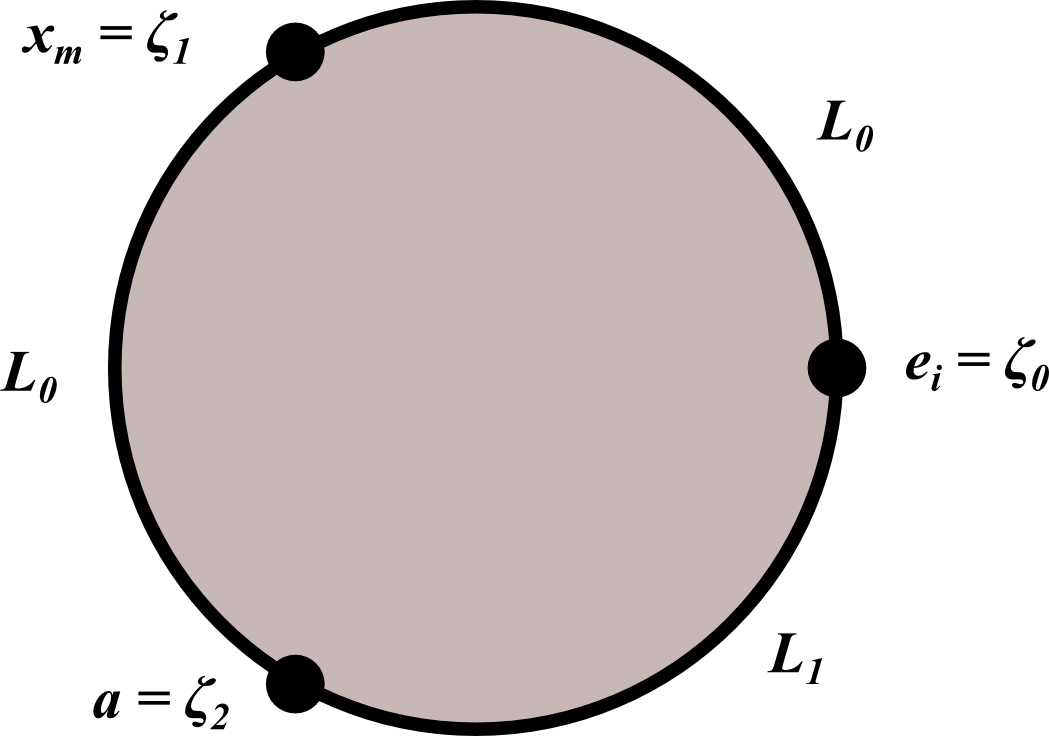}
\caption{Lagrangian labels and asymptotic conditions for marked points of $D$.}
\label{fig:holodisc1}
\end{center}
\end{figure}

We now want to study the product $HF(L_1, L_0) \otimes HF(L_0, L_1) \to HF(L_0, L_0)$. Define this at the chain level by making the following choices: use the same Floer datum for $(L_0, L_1)$ as before, the dual datum for $(L_1, L_0)$, and, for $(L_0, L_0)$, the datum \emph{dual} to the original datum $(H_{L_0, L_0}, J_{L_0, L_0})$. 
 With our new choice of datum, the two generators for $CF(L_0, L_0)$ are still $x_M$ and $x_m$, but now $x_M$ corresponds to $\e$ (for it is the minimum of $-h$), and $x_m$ to $1_{L_0}$.
 
Let $\D'$ be another copy of the closed unit disc in $\C$ with the third roots of unity removed, labelled anticlockwise by $\zeta_0'$, $\zeta_1'$ and $\zeta_2'$. As with $\D$, $\zeta_0'$ is considered an incoming point, and $\zeta_1'$ and $\zeta_2'$ outgoing ones. Equip $\D'$ with the Lagrangian labels as in figure \ref{fig:holodisc2}.

\begin{figure}[htb]
\begin{center}
\includegraphics[scale=0.85]{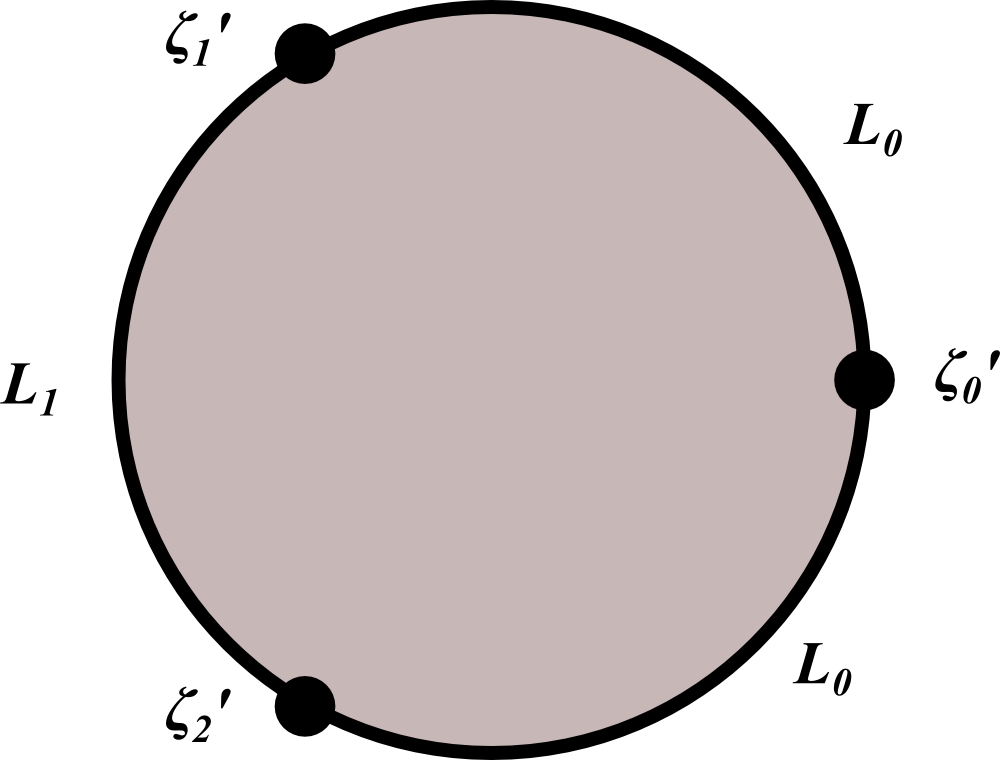}
\caption{Lagrangian labels for $D'$.}
\label{fig:holodisc2}
\end{center}
\end{figure}

Rotation by $2\pi /3$ gives a map $\D' \to \D$ such that $\zeta'_i \mapsto \zeta_{i+1}$, with indices taken mod 3.
Post- and pre-composing with rotation give strip-like ends and a perturbation datum for $\D'$, using those for $\D$ (when equipped with Lagrangian labels as in figure \ref{fig:holodisc1}). We have chosen our Floer data precisely for these to be compatible with it.
Let $\tilde{\F}$ be a Fukaya category defined using these and additional admissible auxiliary data.

 Denote by $e_i^\text{v}$ the generator of  $CF_{\tilde{\F}}(L_1, L_0)$ dual to $e_i$.
The `obvious' linear pairing defined by
\begin{eqnarray}
CF_{\tilde{\F}}(L_1, L_0)  \otimes  CF_{\tilde{\F}}(L_0, L_1) & \to & \Z_2 \\
e_i^\text{v}  \otimes  e_i \qquad \quad \, \,\, \,   \quad  & \to & 1
\end{eqnarray}
descends to a non-degenerate pairing on cohomology, say $(\cdot, \cdot)$. 
Say $c = c_1 + \ldots + c_k$, where each $c_i$ corresponds to one point of $L_0 \pitchfork \phi^1(L_1)$.  Let $c_i^\text{v}$ be the generator of $CF_{\tilde{\F}}(L_1, L_0)$ dual to $c_i$. Pick a class $d \in HF(L_1, L_0)$ such that $(d, c) = 1$. Note that in general, $d$ need not be unique. Fix a chain level representative for $d$, which we also denote by $d$. It must be of the form
\bq
d = c_{k_1}^\text{v} + \ldots + c_{k_l}^\text{v} + f_1^\text{v} + \ldots + f_j^\text{v}
\eq
where $\{ k_1, \ldots , k_l  \} \subset \{ 1 , \ldots , k \}$ and $l$ is \emph{odd}, and each $f_i$ is a chain element corresponding to a point of $L_0 \pitchfork \phi^1(L_1)$ disjoint from the $c_i$, with $f_i^\text{v}$ denoting its dual.

 We want to compute $\mu^2(d,a)$. By construction, the maps $u \in C^\infty (\D, M)$ counted in the products $CF_\F(L_0, L_1) \otimes CF_\F(L_0, L_0) \to CF_\F (L_0, L_1)$ are in one-to-one correspondance with the maps $v \in C^\infty (\D, M)$ counted in the products $CF_{\tilde{\F}}(L_1, L_0) \otimes CF_{\tilde{\F}}(L_0, L_1) \to C_{\tilde{\F}}(L_0, L_0)$. Suppose the asymptotics conditions for $u$ are those of figure \ref{fig:holodisc1}. Then $v$ will satisfy asymptotic conditions as follows:
\begin{itemize}
\item $\zeta'_0$, an incoming point, corresponds to $x_m$, which has cohomology class $1_{L_0}$.
\item $\zeta'_1$, which is an outgoing point, as is $\zeta_2$, corresponds to $a$.
\item $\zeta'_2$ is an outgoing point, whereas $\zeta_0$ is an incoming one; this means that $\zeta'_2$ corresponds to $e_i^\text{v}$. 
\end{itemize}
The mod 2 count of these discs is one precisely when $e_i$ is one of the $c_i$. Thus, as $l$ is odd, the coefficient of $1_{L_0}$ in  $\mu^2(d, a)$ is one.
 This means that either $d \cdot a = 1$ or $d \cdot a = 1+\e$. If the former holds, we are done; if the later does, simply note that $(d + \e \cdot d) \cdot a = 1+ \e + \e \cdot (1+\e) = 1$.  This completes the proof of  Proposition \ref{th:multiplybye}.
\end{proof}

\begin{remark} Although we took a different approach for technical reasons, the informed reader might want to think of the above as a formal consequence of the Frobenius property 
of the Fukaya category (see e.g. \cite[Section 8c]{Seidel08} for statements), together with the ring structure of $HF(L_0, L_0)$. Here is an outline:

For any Lagrangian $L$, we have a linear map
\bq
\int_{L}: \,HF(L,L) \to \Z_2
\eq
which under the PSS isomorphism corresponds to integration over the fundamental class. Moreover, for any two Lagrangians $L_0$ and $L_1$, the pairing
\bq
\xymatrix{
HL(L_1, L_0) \otimes HF (L_0, L_1) \ar[r]^-{\mu^2} &  HF (L_0, L_0) \ar[r]^-{\int_{L_0}} & \Z_2
}
\eq
is non-degenerate (Frobenius property). One could check that this is the same pairing as the one we used in the proof. In this case, knowing that it factors through the Floer product readily provides an element $d \in HF(L_1, L_0)$ such that $d \cdot c \neq 0 \in HF(L_0, L_0)$. By associativity of Floer products, we have
\bq
(d \cdot a) \cdot \e \neq 0 \in HF(L_0, L_0).
\eq 
$L_0$ is a sphere; in particular, we completely understand the ring structure of $HF(L_0,L_0)$. Considering all possibilities, we see that $d \cdot a = 1$ or $d \cdot a = 1+\e$. Now conclude as before.
\end{remark}


\begin{cor}\label{th:multiplybyenontrivialimage}
Suppose $L_0$, $L_1$ are Lagrangians, with $L_0$ a sphere, and
\begin{align}
HF (L_0, L_1) & \to HF (L_0, L_1)  \\
[a] & \mapsto [\mu_2(a,\e)]
\end{align}
has non-trivial image. Then $L_0$ and $L_1$ are Fukaya isomorphic.
\end{cor}

\begin{proof}
This is immediate from Proposition \ref{th:multiplybye} and Lemma \ref{th:isojustneedonedirection}.
\end{proof}

\section{Classification of finite dimensional $A$-modules}\label{sec:classification}

Let us start by recalling some definitions \cite[Section 4.2]{Keller}.

\begin{dd}
Given an $A_{\infty}$ algebra $B$, a right $A_\infty$ module over $B$ is a $\Z_2$ vector space $M$ together with a sequence of maps
\bq
\mu_M^n: M \otimes B^{\otimes (n-1)} \to M
\eq
for all $n \geq 1$, such that
\bq
\sum \mu_M^{r+t+1}( Id^r \otimes \mu^s \otimes Id^t ) =0
\eq
where the sum is taken over all decompositions $r + s + t =n$, with $r, t \geq 0$ and $s > 0$, and $\mu^s$ denotes $\mu_B^s$ when $r>0$, and $\mu_M^s$ otherwise.

If $B$ is strictly unital, and $\mu^n_M$ vanishes whenever one of the entries from $B$ is its unit, we say that $M$ is strictly unital.
\end{dd}

\begin{dd}
A morphism of right $A_\infty$-modules $f: \,N\to M$ is a sequence of morphisms 
\bq
f_n : N \otimes A^{\otimes(n-1)} \to M
\eq 
for $n \geq 1$ such that for all $n$, we have
\begin{equation}\label{eq:relationformorphism}
\sum f_{r+1+t} (1^{\otimes r} \otimes \mu^s \otimes 1^{\otimes t}) = \sum \mu_{j+1}(f_l \otimes 1^{\otimes j}) 
\end{equation}
where the left hand sum is taken over all decompositions $n=r+s+t$, with $r,t\geq 0$, $s\geq 1$, and the right-hand sum is taken over all decompositions $n=j+l$, $j \geq 0$, with $l \geq 1$.
The morphism is said to be strictly unital if $f_n$ ($n \geq 2$) vanishes whenever one of the entries for $A$ is the unit 1. 
\end{dd}
The definitions for left modules are analogous.
\newline

Recall $A$ is the strictly unital  $A_\infty$ algebra $\mathbb{Z}_2 [\epsilon] / \epsilon^2$
with the following $A_\infty$ structure: $\mu^2(1,\epsilon)=\mu^2(\epsilon,1)=\epsilon$, and all other maps are zero.
We wish to classify strictly unital, finite dimensional $A_\infty$ modules over $A$. For such a right $A_\infty$ module, say $M$, the only (potentially) non-trivial action of $A$ is given by $\mu^1(\cdot)$, $\mu^2(\cdot,\epsilon)$, $\mu^3(\cdot, \epsilon, \epsilon)$, \ldots 
and similarly for left modules. For $k \geq 1$, define $R_k$ to be the right $A$-module with two generators (as a vector space), $r_k^0$ and $r_k^1$, such that $\mu^k(r_k^0,\epsilon,\ldots,\epsilon) = r_k^1$, and the $A_\infty$ action is trivial otherwise. Define $\Lk$, a left $A$-module with generators $l_k^0$ and $l_k^1$, similarly.

\begin{prop}\label{th:classificationofmodules}
Let $M$ be a strictly unital, finite dimensional right (resp. left) $A_\infty$ module over $A$. Then $M$ is quasi-isomorphic to a finite dimensional module $N$, that decomposes as a direct sum of $A_\infty$ modules of the following forms:
\begin{itemize}
\item copies of $\mathbb{Z}_2$, with the trivial $A_\infty$ action;
\item finitely many $R_k$'s (resp. $\Lk$'s).
\end{itemize}
\end{prop}
We present a proof for right-modules; as $A$ is isomorphic to its opposite algebra, the case of left-modules follows. Let $t$ be a formal variable; define a differential on $M[[t]]$ by setting
\begin{equation}\label{eq:differential}
d(a)= \mu^1(a) + t\mu^2(a, \epsilon) + \ldots + t^n \mu^{n+1} (a, \epsilon, \ldots, \epsilon) + \ldots
\end{equation}
for all $a\in M$, and extending this linearly in $t$ to $M[[t]]$. The $A_\infty$ relations for $M$ imply that $d^2 =0$.
Consider $H(M[[t]],d)$; as $M$ is finite (as a $\mathbb{Z}_2$-vector space), this is a finitely generated $\mathbb{Z}_2[[t]]$-module. Since $\Z_2[[t]]$ is a P.I.D., the standard decomposition theorem gives:
\bq
H(M[[t]],d) \cong \mathbb{Z}_2 [[t]] \oplus \ldots \oplus \mathbb{Z}_2[[t]] \oplus \mathbb{Z}_2[[t]] / (t^{k_1}) \oplus \ldots \oplus \mathbb{Z}_2[[t]] / (t^{k_n})
\eq
for some positive integers $k_i$. 

Our strategy is as follows: using this description, we construct a finite dimensional $A_\infty$ module $N$, with structure as described in the proposition, together with a quasi-isomorphism 
\bq
(N[[t]],d) \to (M[[t]],d).
\eq
We then check that the map this induces between $N$ and $M$ is a quasi-isomorphism of $A_\infty$ modules.

\vspace{1pc}

Let $r$ be the number of free summands in the direct sum decomposition of $H(M[[t]], d)$. We define $N$ to be an $A$-module with $r+2n$ generators (as a $\mathbb{Z}_2$-vector space),  given by the direct sum of the following $A$-modules: $r$ copies of $\mathbb{Z}_2$, and $R_{k_i}$ for $i = 1, \ldots, n$.
Now consider $N[[t]]$, with differential $d$ defined as in equation \eqref{eq:differential}. By construction, $d$ is zero on the first $r$ copies of $\mathbb{Z}_2[[t]]$, 
$d(r_{k_i}^0) = t^{k_i}r_{k_i}^1$, and $d(r_{k_i}^1) = 0$. 
There is an obvious isomorphism of $\Z_2[[t]]$-modules $\tilde{f}: H(N[[t]],d) \cong H(M[[t]],d)$. $N[[t]]$ is a free $\Z_2[[t]]$-module; we lift $\tilde{f}$ to a quasi-isomorphism $f: N[[t]] \to M[[t]]$ of $\Z_2[[t]]$-modules, by specifying the image of our given basis:
\begin{itemize}
\item Let $e_i$, $1 \leq i \leq r$, be a generator for the $i^{th}$ $\Z_2[[t]]$ summand in $N[[t]]$. This is naturally a generator for the $i^{th}$ $\Z_2[[t]]$ summand of $H(N[[t]],d)$; map it to any lift in $M[[t]]$ of its image under $\tilde{f}$.
\item  Map $r_{k_i}^1$ to any lift in $M[[t]]$ of the generator of $\Z_2[[t]] / (t^{k_i})$, say $\rho_i^1$. Then $t^{k_i}\rho_i^1=0 \in H(M[[t]],d)$; pick $\rho_i^0$ such that $t^{k_i} \rho_i^1 = d \rho_i^0$ in $M[[t]]$. Now set $f(r_{k_i}^0) = \rho_i^0$.
\end{itemize}
This gives a well-defined chain map, that, by construction, is an isomorphism on cohomology.

Let $\Pi_n: M[[t]] \to M$ be the projection to the coefficient of $t^n$ in the power series. We define the maps $f_n$ by setting
\begin{displaymath}
f_n(a, \overbrace{\epsilon, \ldots, \epsilon}^{n-1}) = \Pi_{n-1} (f(a))
\end{displaymath}
and requiring $f$ to be strictly unital. We must check that this is indeed a map of $A_\infty$ modules. This essentially follows from the fact that $f: N[[t]] \to M[[t]]$ was a chain map, though one has to pay attention to some details. Let us spell these out.

Fix $n$, and consider $f_{n+1}: N \otimes A^{\otimes n} \to M$. First suppose you set all entries for $A$ to $\e$. Note that for a summand on the left-hand side of equation \eqref{eq:relationformorphism} to be non-zero, we need $r=0$. (At this point the $A_\infty$ structure on $hom(L,L)$ is crucial: knowing the multiplicative structure for cohomology does not suffice.) As $f$ is a chain map, for all $a \in N$, $df(a)=fd(a) \in M[[t]]$. Taking $\Pi_n$, this precisely tells us that equation \eqref{eq:relationformorphism} is satisfied. 
It remains to consider the case where some of the entries for $A$ are set to $1$. From strict unitality, equation \eqref{eq:relationformorphism} holds trivially unless there is exactly one entry equal to $1$, and all others are $\e$. Let us examine the left-hand side of equation \eqref{eq:relationformorphism}. If $1$ is in the final position, the only (potential) non-zero terms are
\begin{itemize}
\item for $r=0$, $f_1(\mu_2(a,1)) = f_1(a)$ for $n=2$;
\item for $r \neq 0$, $f_l(a, \e, \ldots, \e, \mu_2(\e, 1)) = f_l(a, \e, \ldots, \e)$ for $n=l+1 \geq 3$.
\end{itemize}
If $1$ is not in the final position, there are exactly two non-zero terms, both for $s = 2$. One involves $\mu_2(1, \e)$, and the other $\mu_2(\e,1)$ or $\mu_2(a,1)$, and, after contracting, they cancel out.
The right-hand side of equation \eqref{eq:relationformorphism} can only be non-zero when $1$ is in the final position, in which case we get
\bq
\mu_2 (f_l (a, \e, \ldots, \e), 1) = f_l (a, \e, \ldots, \e) 
\eq
for $n=l+1 \geq 2$.
Thus the map we defined is indeed a morphism of $A_\infty$-modules.

Finally, we show that $f: N \to M$ is a quasi-isomorphism. Consider the two short exact sequences of chain complexes, fitting into a commutative diagram:

\bq
\xymatrix{
0 \ar[r] & N[[t]] \ar[r]^{\cdot t} \ar[d]_f & N[[t]] \ar[r]  \ar[d]_f & N \ar[r] \ar[d]_{f_1} & 0 \\
0 \ar[r] & M[[t]] \ar[r]^{\cdot t}          & M[[t]] \ar[r]           & M \ar[r]          & 0 }
\eq
Each short exact sequence induces a long exact sequence on cohomology, and the vertical maps in the diagram induce maps between these long exact sequences. As we already know that $f$ is a quasi-isomorphism, the five-lemma tells us that $f_1$ is one, too.

 
\section{Twisted complexes and iterations of Dehn twists}\label{sec:generalisation}

Suppose $M$ is an exact symplectic manifold with contact type boundary, $L$ a Lagrangian sphere, and $L_0$ any exact Lagrangian. This section contains the computation of an expression for  the iterated Dehn twist $\tau^n_L L_0$ as a twisted complex. The starting point is: 

\begin{thm}\label{th:Seidel}(\cite[Theorem 1]{Seidel03} and \cite[Corollary 17.17]{Seidel08})
Suppose $A$ is a representative of the Fukaya category of $M$ that is finite and strictly unital. Then $\tau_L L_0$, the Dehn twist of $L_0$ about $L$, is quasi-isomorphic to 
\bq
Cone (ev: hom(L,L_0) \otimes L \to L_0 )
\eq
as objects of $Tw\A$.
\end{thm}
A basic property of twisted complexes, that gets us off the ground, is:

\begin{lem}\label{qicones}
Suppose $X$ and $Y$ are quasi-isomorphic in $Tw\A$; let $Z$ be any object of $Tw\A$, and
\begin{eqnarray}
ev_X: hom_{Tw\A}(Z,X) \otimes Z \to X \\
ev_Y: hom_{Tw\A}(Z,Y) \otimes Z \to Y
\end{eqnarray}
Then $Cone(ev_X)$ and $Cone(ev_Y)$ are quasi-isomorphic objects of $Tw\A$.
\end{lem}

\begin{proof}
This follows immediately from \cite{Seidel08}, Corollary 3.16.
\end{proof}

This means that one can write down an expression for $\tau^n_L L_0$ as an iteration of cones, and look to simplify it.
\newline

We know that $\A$ is cohomologically finite and c-unital (subsection \ref{sec:Fukayacategory}), and that $CF_{\A}(L,L)$ is quasi-isomorphic to $A$ (Proposition \ref{th:lagspherestructure}). By Proposition \ref{th:nicerepresentative}, we may assume that $\A$ is minimal, finite, strictly unital, and that $hom_\A(L,L)=A$ as $A_\infty$ algebras. We will heavily rely on the the simplicity of this structure to reduce the expression for $\tau_L^n L_0$.

The manipulations of twisted complexes that follow are valid for any finite, strictly unital $A_\infty$ category $\A$, with a distinguished element $L$ such that $hom(L,L) = A$, and $L_0, L_1$ any objects of $\A$. 
Unlike the previous assumptions, it is not crucial that $\A$ be minimal; it will simply make our computations, and the complexes we encounter, cleaner.

\begin{prop}\label{th:bartwistedcomplex}
Let $T^n_L L_0$ denote the twisted complex
\bq L_0 \; \oplus \; hom(L,L_0)\otimes L \; \oplus \; hom(L,L_0)  \otimes ( \e)\otimes L \; \oplus \ldots \oplus \; hom(L,L_0)\otimes \overbrace{(\e) \ldots (\e)}^{n-1} \otimes L
\eq
where the connection is given by 
\begin{itemize}
\item 0 on the first summand;
\item $ev$, the evaluation map $hom(L, L_0) \otimes L \to L_0$  on the second summand;
\item $Id\otimes ev + \mu^2 \otimes 1$ on the third summand, where $ev$ is the evaluation map $\e \otimes L \to L$;
\item \ldots
\item On the $r^{th}$ summand, 
\bq
Id^{\otimes(r-2)} \otimes ev +
\sum (\mu^i \otimes Id^{\otimes j}) \otimes 1  
\eq
where the sum ranges over all decompositions $r-1= i+j$, with $i>1$. Note that $ \mu^i \otimes Id^{\otimes j} \otimes 1$ maps to the $(j+2)^{th}$ summand, and $Id^{\otimes(r-2)} \otimes ev$ maps to the $(r-1)^{th}$ summand.
\end{itemize}
We claim that $T^n_L L_0$ is quasi-isomorphic to $\tau^n_L L_0$.
\end{prop}

\begin{rmk}  $(\e)$ is a one-dimensional vector space; one could also present $T^n_L L_0$ using multiple copies of $hom(L,L_0)\otimes L$ and push the $\e$'s over to the expression for the connection. We will obtain the complex in Proposition \label{th:bartwistedcomplex} as a quotient; from that prospective, the presentation above seemed the most natural one.
\end{rmk}

We shall prove this by induction on $n$. Theorem \ref{th:Seidel} gives the case $n=1$.
We start with a preliminary computation.
\begin{lem}\label{th:mu_1structureonhom}
Fix an object $L_1$ of $\A \subset Tw\A$.  Then
\begin{multline}
hom(L_1, T^n_L L_0) = 
hom(L_1, L_0) \; \oplus \; hom(L,L_0)\otimes hom(L_1, L) \\ 
\oplus \; hom(L,L_0)  \otimes ( \e ) \otimes hom(L_1,L)\; \oplus \ldots \oplus \; hom(L,L_0)\otimes (\e) \ldots (\e) \otimes hom(L_1, L)
\end{multline}
and $\mu^1_{Tw\A}$ acts
\begin{itemize}
\item on the first summand, by zero;
\item on the second summand, by $\mu^2_{\A}$;
\item \ldots
\item on the $r^{th}$ summand, by
\bq
\sum_{r= i+j, j >1} \big( Id^{\otimes i} \otimes \mu^j_\A  +  \mu^j_\A \otimes Id^{\otimes i} \big).
\eq
\end{itemize}
\end{lem}

\begin{proof} By definition,
\bq \mu^1_{Tw\A} (a) = \sum_{i \geq 0} \mu^{1+i}_{\Sigma\A} ( \delta_{T^n_L L_0} , \ldots,
\delta_{T^n_L L_0}, a)
\eq
After expanding, the claim boils down to the $A_\infty$-associativity equations for $hom(L,L_0)$ together with Lemma  \ref{th:evmaps}.  Let us do the calculation for the third summand. Consider a class $c = a \otimes \e \otimes b$ where $a \in hom(L, L_0)$ and $b \in hom (L_1, L)$. We have
\begin{eqnarray}
\mu^1_{Tw\A} (a \otimes \e \otimes b) & = & 
\mu^1_{\Sigma \A} (a \otimes \e \otimes b) 
+ \mu^2_{\Sigma \A} (Id \otimes ev + \mu^2 \otimes 1, a \otimes \e \otimes b) \nonumber \\&&
+ \mu^3_{\Sigma \A} (ev, Id \otimes ev + \mu^2 \otimes 1, a \otimes \e \otimes b) \\
& = & 0 + 
\mu^2_{\Sigma \A} (Id \otimes ev , a \otimes \e \otimes b)
+ \mu^2_{\Sigma \A} ( \mu^2 \otimes 1, a \otimes \e \otimes b) \nonumber \\&&
+ \mu^3_{\Sigma \A} (ev,  \mu^2 \otimes 1, a \otimes \e \otimes b) 
+ \mu^3_{\Sigma \A} (ev, Id \otimes ev, a \otimes \e \otimes b)
\\
& = & 
a \otimes \mu^2 (\e ,b)
+ \mu^2(a , \e) \otimes b + 0 
+ \mu^3 (a, \e, b)
\end{eqnarray}
where we made use of the minimality of the $A_{\infty}$ structure of $\A$ at the second step, and of its strict unitality at the third.
\end{proof}
Suppose the claim in Proposition \ref{th:bartwistedcomplex} is true for all powers of $\tau_L$ up to $\tau^{n-1}_L$. By Theorem \ref{th:Seidel} and Lemma \ref{qicones}
\begin{align}
\tau^{n}_L (L_0) & \cong Cone(ev: hom(L,T^{n-1}_L L_0) \otimes L \to T^{n-1}_L L_0) \\
& = T^{n-1}_L L_0   \oplus hom(L, T^{n-1}_L L_0)\otimes L
\end{align}
with the cone connection.

Lemma \ref{th:mu_1structureonhom} provides us with an expression for the differential on  $hom(L, T^{n-1}_L L_0)$, and thus for the cone connection on 
\bq
T^{n-1}_L L_0 \oplus hom(L, T^{n-1}_L L_0)\otimes L
\eq solely in terms of the composition maps of $\A$. Let us denote this complex by $\widehat{T^{n}_L} L_0$. 
Our strategy is as follows: by repeated quotienting of subcomplexes, we exhibit a complex that is quasi-isomorphic to $\widehat{T^n_L} L_0$, and of the desired form. Each subcomplex we shall quotient out by will be the direct sum of two pieces between which the connection gives an isomorphism, ensuring that quotienting does not change what element of $H(Tw\A)$ we have (see Lemma \ref{th:isobetweentwosummands}).
This is chiefly a matter of book-keeping; we start by presenting the cases $n=2$ and $3$, in the hope that they render the general computation less forbidding.
\vspace{1pc}

The twisted complex $\widehat{T^2_L} L_0$ is given by

\begin{equation} \xymatrix@C=+2em{
hom(L,L_0)\otimes hom(L,L) \otimes L
\ar[rr]_-{Id \otimes ev} 
\ar[d]^{\mu^2 \otimes 1} 
& & 
hom(L, L_0) \otimes L \ar[d]_{ev} 
 \\ hom(L,L_0) \otimes L  \ar[rr]_-{ev}
 & &  L_0
}
\end{equation}
where the complex is the direct sum of the four terms in the diagram, and the maps give the connection $\delta$ (we will suppress subscripts on connections for the rest of this section). 
$T^{n-1}_L L_0$ is on the right column, and $hom(L, T^{n-1}_L L_0) \otimes L$ on the left one.
Consider $hom(L,L_0) \otimes 1 \otimes L$; its image under $\delta$ is the diagonal copy $\Delta$ of $hom(L,L_0) \otimes L$; moreover, $hom(L,L_0) \otimes 1 \otimes L$, together with its image, form a subcomplex. Quotient this out to get the complex
\begin{equation}
hom(L,L_0) \otimes \e \otimes L \xrightarrow{\mu^2 \otimes 1 + Id \otimes ev\,\,\,} hom(L,L_0) \otimes L  \xrightarrow{\quad ev \quad} L_0
\end{equation}
Further, note that $\delta$ gave an isomorphism 
\bq
hom(L,L_0) \otimes 1 \otimes L \xrightarrow{(\mu^1 \otimes 1, Id \otimes ev)}
\Delta \subset hom(L,L_0) \otimes L \; \times \; hom(L,L_0) \otimes L
\eq
Thus the new complex is quasi-isomorphic to the old one, using Lemma \ref{th:isobetweentwosummands}.
\vspace{1pc}

Can this be generalised? Let us look at $n=3$. We get:

\begin{equation} \xymatrix@C=+2em{
hom(L, L_0) \otimes \e \otimes hom(L,L) \otimes L  
\ar[d]^{\mu^2\otimes Id \otimes 1+ Id \otimes \mu^2 \otimes 1}
\ar@/^-7pc/[dd]_{\mu^3\otimes 1}
 \ar[rr]_-{Id^{\otimes 2}\otimes ev} & & 
hom(L,L_0) \otimes \e \otimes L 
\ar[d]^{\mu^2 \otimes 1 + Id \otimes ev}
 \\
 hom(L,L_0) \otimes hom(L,L) \otimes L 
\ar[rr]_-{Id \otimes ev} 
\ar[d]^{\mu^2 \otimes 1} & & 
hom(L, L_0) \otimes L \ar[d]_{ev}
 \\
hom(L,L_0) \otimes L  \ar[rr]_-{ev} & &  L_0
}
\end{equation}
As before, the image of $hom(L,L_0) \otimes 1 \otimes L$ under $\delta$ is the diagonal copy of $hom(L,L_0) \otimes L$, and 
\begin{itemize}
\item $hom(L,L_0) \otimes 1 \otimes L$ and its image under $\delta$ form a subcomplex.
\item $\delta$ gives an isomorphism of $hom(L, L_0) \otimes 1 \otimes L$ onto its image.
\end{itemize}
The quotient complex, quasi-isomorphic to the old one, is:

\begin{equation} \xymatrix@C=+2em{
hom(L, L_0) \otimes \e \otimes hom(L,L) \otimes L  
\ar[d]_{\mu^2\otimes Id\otimes 1+ Id \otimes \mu^2 \otimes 1}
\ar[dr]_{\mu^3\otimes 1}
 \ar[r]_-{Id^{\otimes 2}\otimes ev} &
hom(L,L_0) \otimes \e \otimes L 
\ar[d]^{\,\, \mu^2 \otimes 1 + Id \otimes ev}
 \\
 hom(L,L_0) \otimes \e \otimes L 
\ar[r]_{\mu^2 \otimes 1 + Id \otimes ev \,\,\,}  
& hom(L,L_0) \otimes L  \ar[d]_-{ev}   \\
&  L_0
}
\end{equation}
Now, consider the image under $\delta$ of $hom(L,L_0) \otimes \e \otimes 1 \otimes L$; this is the diagonal copy of $hom(L,L_0) \otimes \e \otimes L$. ($\mu^2 \otimes Id\otimes 1$ has no image by construction, and $\mu^3 \otimes 1$ has none because of unitality.) Again, $hom(L,L_0) \otimes \e \otimes 1 \otimes L$ and its image form a subcomplex, and $\delta$ gives an isomorphism between them. Quotienting gives the quasi-isomorphic twisted complex
\begin{equation}
\xymatrix@C=7.2em{
hom(L,L_0) \otimes \e \otimes \e \otimes L 
\ar[r]^-{\mu^2 \otimes Id \otimes 1+ Id^{\otimes 2} \otimes ev }
\ar@/_3pc/[rr]^{\mu^3 \otimes 1}
& hom(L,L_0) \otimes \e \otimes  L 
\ar[r]^-{\mu^2 \otimes 1 + Id \otimes ev}
&
hom (L,L_0)\otimes L 
\xrightarrow{ev}
L}
\end{equation}
which we recognise as $T^3_L L_0$. (As $\mu^2(\e, \e)=0$, the $Id \otimes \mu^2 \otimes 1$ term has vanished. See also remark \ref{rmk:weakerhypothesis}.)
\vspace{1pc}

\noindent
For a general $n$, proceed similarly:
\begin{itemize}
\item First quotient out $hom(L,L_0) \otimes 1 \otimes L$ and its image under $\delta$, the diagonal copy of $hom(L,L_0) \otimes L$; the new complex has one copy of $L_0$, one of $hom(L,L_0)\otimes L$, two of $hom(L,L_0) \otimes \e \otimes L$, and terms with higher tensor length.
\item Then quotient out $hom(L,L_0) \otimes \e \otimes 1 \otimes L$ and its image under $\delta$, which is the diagonal copy of $hom(L,L_0) \otimes \e \otimes L$; the new complex agrees with $T^n_L L_0$ for terms containing up to two tensor signs, has two copies of $hom(L,L_0) \otimes \e \otimes \e \otimes L$, and terms with a higher number of tensor signs.
\item \ldots
\item At the $r^{th}$ stage, we quotient out $hom(L, L_0) \otimes \overbrace{\e \otimes \ldots \otimes \e}^{r-1} \otimes 1 \otimes L$, and its image under $\d$. Notice that until this point, $hom(L, L_0) \otimes \overbrace{\e \otimes \ldots \otimes \e}^{r-1} \otimes 1 \otimes L$ was unaffected by the quotienting process.
\item \ldots
\item Finally, quotient out $hom(L,L_0) \otimes \overbrace{ \e \otimes \ldots \e}^{n-2} \otimes 1 \otimes L$ and its image under $\delta$, the diagonal copy of  $hom(L,L_0) \otimes \overbrace{ \e \otimes \ldots \e}^{n-2} \otimes L$. This gives the complex $T^n_L L_0$.
\end{itemize}
This completes the proof of Proposition \ref{th:bartwistedcomplex}.

\begin{remark}\label{rmk:weakerhypothesis}
Suppose that we started with weaker hypothesis: instead of having $hom(L,L) = A$, we only had $hom(L,L) = B$, where $B = \Z_2(\e)/(\e^2)$ is a minimal, strictly unital $A_\infty$ algebra with $\mu^2(1, \e) = \mu(\e, 1) = \e$, but some of the higher compositions might be non-zero. (The other assumptions on $\A$ are unchanged.) Then one could check that $\tau^n_L L_0$ is quasi-isomorphic to 
\bq L_0 \; \oplus \; hom(L,L_0)\otimes L \; \oplus \; hom(L,L_0)  \otimes  \e\otimes L \; \oplus \ldots \oplus \; hom(L,L_0)\otimes \overbrace{\e \ldots \e}^{n-1} \otimes L
\eq
with connection acting on the $r^{th}$ summand by 
\bq
Id^{\otimes r-2} \otimes ev \; \oplus \; \sum_{r = i+j+k, \, j>1} Id^{\otimes i} \otimes \mu^j_\A \otimes Id^{\otimes k-1} \otimes 1.
\eq
\end{remark}


\section{Lower bounds on $hf(\tau_L^n L_0, L_1)$}\label{sec:lowerbounds}


\subsection{$T^n_L L_0$ as a cone}\label{sec:T^n_LL_0asacone}
The twisted complex $T^n_L L_0$ of Proposition \ref{th:bartwistedcomplex} can itself be viewed as a cone, as follows.

\begin{definition}
 Let $C^n_L L_0$ be the twisted complex
\begin{multline}
hom(L,L_0)\otimes L 
 \, \oplus \, hom(L,L_0)\otimes \e \otimes L 
\, \oplus \, hom(L,L_0)\otimes \e \otimes \e \otimes L 
\, \oplus \, \ldots \\
\, \oplus \, hom (L,L_0) \otimes \e \otimes \ldots \otimes \e \otimes L 
\end{multline}
with connection \bq Id^{k-2} \otimes ev + \sum_{k-1 = i+j, \, i>1} \big(\mu^i \otimes Id^j \big) \otimes 1 
\eq
on the $k^{th}$ summand, $k>0$. (This is a quotient of $T^n_L L_0$.)
\end{definition}
There is a map of twisted complexes $C^n_L L_0 \to L_0$ given by $ev: hom(L,L_0) \otimes L \to L_0$  on the first summand, and zero elsewhere; by construction, $T^n_L L_0$ is the cone of this map. 

\begin{lem}\label{th:homstructureforC_n}
Suppose $L_1 \in \text{Ob}{\A} \subset \text{Ob} Tw(\A)$. Then
\begin{multline}
hom(L_1, C^n_L L_0) = 
hom(L,L_0)\otimes hom(L_1, L) 
 \, \oplus \, hom(L,L_0)\otimes \e \otimes hom(L_1,  L) \\
\, \oplus \, hom(L,L_0)\otimes \e \otimes \e \otimes hom(L_1, L) 
\, \oplus \, \ldots
\, \oplus \, hom (L,L_0) \otimes \e \otimes \ldots \otimes \e \otimes hom(L_1, L) 
\end{multline}
and $\mu^1_{Tw\A}$ acts on the $k^{th}$ summand by
\bq
\sum_{i+j = k+1, \,i, j >0} \big(\mu^i \otimes Id^{j}  + Id^{i} \otimes \mu^j \big)
\eq
\end{lem}

\begin{proof}
This is essentially the same computation as for Lemma \ref{th:mu_1structureonhom}.
\end{proof}
Notice this is the reduced bar complex for the tensor $hom(L, L_0) \otimes_{A} hom(L_1, L)$, truncated (see e.g.~\cite[Section 2.3.3]{Hasegawa}).


\subsection{Quasi-isomorphic truncated bar complexes}

Let $M$ be a right $A$-module, and $N$ a left $A$-module; assume moreover that they are both finite and strictly unital. We will use the following notation for the truncated reduced bar complex of their tensor product
\bq (M \otimes_A N)_n := M \otimes N  \, \oplus \,
M \otimes \e \otimes N \, \oplus \,
\ldots \, \oplus \,
M \otimes \underbrace{\e \otimes \ldots \otimes \e}_{n-1} \otimes N
\eq
with the natural differential induced by the module structures on $M$ and $N$: on the $r^{th}$ summand, it is
\bq
\sum_{r-1=i+j, \, i, j >0 }\big( \mu^i \otimes Id^{\otimes j} + Id^{\otimes j} \otimes \mu^i   \big)
\eq
($\mu^1_{Tw{\A}}$ on $hom(L_1, C^n_L L_0)$, in Lemma \ref{th:homstructureforC_n}, is an example of this.)
Let $H(M\otimes_AN)_n$ denote the cohomology of this chain complex. We shall use the following:

\begin{lem}\label{th:changeqiclass} Let $M'$ (resp. $N'$) be a finite, stricly unital right (resp. left) $A$-module.
Suppose that $M'$ is quasi-isomorphic to $M$, and $N'$ quasi-isomorphic to $N$. Then there is a quasi-isomorphism 
\bq h:(M'\otimes_A N')_n \to (M\otimes_A N)_n
\eq
\end{lem}

\begin{proof}
Suppose $f: M' \to M$ is a quasi-isomorphism of left $A$-modules, given by a collection of maps
\bq f_r: M'\otimes A^{\otimes(r-1)} \to M
\eq
for all $r\geq 1$. By Lef\`evre-Hasegawa (\cite[Theorem 3.2.2.1]{Hasegawa}), we may assume $f$ is a strictly unital map. Let $\tilde{f}: (M'\otimes_A N')_n \to (M \otimes_A N')_n$ be the `obvious' map constructed using all of the $f_r$: $\tilde{f}$ acts on 
\bq M' \otimes \underbrace{\e \otimes \e \ldots \otimes \e}_{k} \otimes N '
\eq
by
\bq \tilde{f}(a \otimes \e \otimes \ldots \otimes \e \otimes b) = \sum_{i = 1}^{k+1} f_i (a, \e, \ldots, \e) \otimes \underbrace{\e \otimes \ldots \otimes \e}_{k-i+1} \otimes  b.
\eq
$\tilde{f}$ is a chain map (this follows from the equations satisfied by $f$ as an $A$-module morphism).
Similarly, starting with a strictly unital quasi-isomorphism 
\bq g_r: A^{\otimes (r-1)} \otimes N' \to N
\eq
contruct a chain map
\bq \tilde{g}: (M \otimes_A N')_n \to (M \otimes_AN)_n 
\eq
Filter each of the complexes by length of tensor product, for instance:
\bq (M \otimes_A N)_n \supset (M \otimes_A N)_{n-1} \supset \ldots
\supset (M \otimes_A N)_1=M \otimes N \supset 0.
\eq
$\tilde{f}$ respects the filtrations;  it therefore induces a sequence of maps on the associated gradeds, that, by construction, are quasi-isomorphisms; thus, for instance by an inductive application of the 5-lemma, $\tilde{f}$ is a quasi-isomorphism. Similarly, $\tilde{g}$ is a quasi-isomorphism.
\end{proof}


\subsection{Rank Inequalities}

We are now ready to prove the weak analogue of the inequalities used by Ishida. 

\begin{prop}\label{th:rankmasterinequalities}
For any two Lagrangians $L_0$ and $L_1$, and any Lagrangian sphere $L$, and all integers $n \neq 0$
\bq
hf(\tau^n_L (L_0), L_1) + hf(L_0, L_1) \geq hf (L, L_1) \cdot hf (L_0, L).
\eq
Further, if $L \ncong L_0, L_1$ in the Fukaya category, and $|n| \geq 2$, 
\bq
hf(\tau^n_L (L_0), L_1) + hf(L_0, L_1) \geq 2hf (L, L_1) \cdot hf (L_0, L).
\eq
\end{prop}

\begin{proof}
First note that
\bq
hf(\tau^n_L L_0, L_1) = hf (L_0, \tau^{-n}_L L_1) = hf (\tau^{-n}_L L_1, L_0)
\eq
so it is enough to prove the claims for $n>0$.
From our description of $T^n_L L_0$ as a cone (section \ref{sec:T^n_LL_0asacone}), we know we have an exact sequence
\begin{equation} 
\ldots \to
    {Hom_{Tw\A}( C^n_L L_0, L_1)} \to {Hom_{Tw\A}(L_0, L_1)} \to
  {Hom_{Tw\A}(\tau^n_L L_0, L_1)} \to \ldots
\end{equation}
Taking ranks gives
\bq hf (\tau^n_L L_0, L_1) + hf(L_0,L_1) \geq rk \big( Hom_{Tw\A}(C^n_L L_0, L_1) \big)
\eq
so it is enough to show that 
\bq rk \big( Hom_{Tw\A}(C^n_L L_0, L_1) \big) \geq hf(L,L_1) \cdot hf(L_0,L_1).
\eq
By Lemma \ref{th:homstructureforC_n},
\bq
 rk \big( Hom_{Tw\A}(C^n_L L_0, L_1) \big) = rk \big( H(hom(L,L_0)\otimes_A hom(L_1, L))_n \big).
\eq
By Lemma \ref{th:changeqiclass}, we are free to use any representatives of the quasi-isomorphism classes of $hom(L,L_0)$ and $hom(L_1,L)$.
Replace them by modules of the form described by Proposition \ref{th:classificationofmodules}, say $M$ and $N$. For simplicity, assume $M$ and $N$ are minimal.
The chain complex $(M\otimes_AN)_n$ decomposes as a direct sum of complexes of the form
\begin{itemize}
\item $(\Z_2 \otimes_A \Z_2)_n$ \qquad (a)
\item $(R_k \otimes_A \Z_2)_n$ \qquad (b)
\item $(\Z_2 \otimes_A \Lk)_n$ \qquad (b')
\item $(R_j \otimes_A \Lk)_n$ \qquad (c)
\end{itemize}
obtained from all the possible pairings of a summand of $M$ and a summand of $N$.
We want to show that 
\bq rk \; H((M \otimes_A N)_n) \geq rk\;M \times rk\; N
\eq
It is enough to prove this for each of the direct summands itemized above. 
\newline

\noindent
{\bf (a)} The differential on $(\Z_2 \otimes_A \Z_2)_n$ is zero, so its cohomology has rank $n$.
\newline

\noindent
{\bf(b \& b')} The cases $(R_k \otimes_A \Z_2)_n$ and $(\Z_2 \otimes_A \Lk)_n$ are clearly symmetric; without loss of generality, let us consider the first one. Let $u$ be the generator of $\Z_2$.

The elements $r_k^0 \otimes u$ and $r_k^1 \otimes \underbrace{\e \otimes \ldots \otimes \e}_{n-1} \otimes u$ survive passing to cohomology and give distinct classes. 
Suppose additionally that $n \geq 2$ and $L \ncong L_0, L_1$ in the Fukaya category. By Corollary \ref{th:multiplybyenontrivialimage}, we must have $k\geq 3$. In this case, the generators $r_k^0 \otimes \e \otimes u$ and $r_k^1 \otimes \underbrace{\e \otimes \ldots \otimes \e}_{n-2} \otimes u$ also survive; the four classes are linearly independent.
\newline

\noindent
{\bf(c)} We want to show that the cohomology of $(R_j \otimes_A \Lk)_n$ has rank at least 4. Without loss of generality, $j \leq k$. 

$r_j^0\otimes l_k^0$ and $r_j^1 \otimes \underbrace{\e \otimes \ldots \otimes \e}_{n-1} \otimes l_k^1$ give two classes. 

\begin{itemize}
\item If $n\geq k$, $$d(r_j^1 \otimes \underbrace{\e \ldots \e}_{n-1} \otimes l^0_k)=d(r_j^0 \otimes \underbrace{\e \ldots \e}_{n-1-k+j} \otimes l_k^1) = r_j^1 \otimes \underbrace{\e \ldots \e}_{n-k} \otimes l_k^1.$$
Thus $r_j^1 \otimes \e \ldots \e \otimes l^0_k + r_j^0 \otimes \e \ldots \e \otimes l_k^1$, the sum of the elements above, is in the kernel; it gives a non-zero cohomology class. Also,

\bq
d(r_j^0 \otimes \underbrace{\e \ldots \e}_{k-1} \otimes l_k^0) = r_j^0 \otimes l_k^1 + r_j^1 \otimes \underbrace{\e \ldots \e}_{k-j} \otimes l_k^0
\eq
both summands of which lie in the kernel of $d$. Under $d$, they are only the image of $r_j^0 \otimes \e \ldots \e \otimes l_k^0$. This gives us the final class we wanted for the cohomology. See figure \ref{fig:j2k3n3} for the case $j=2$, $k=3$ and $n=3$.

\begin{figure}[htb]
\begin{center}
\includegraphics[height=1.6in, width=4.5in,angle=0]{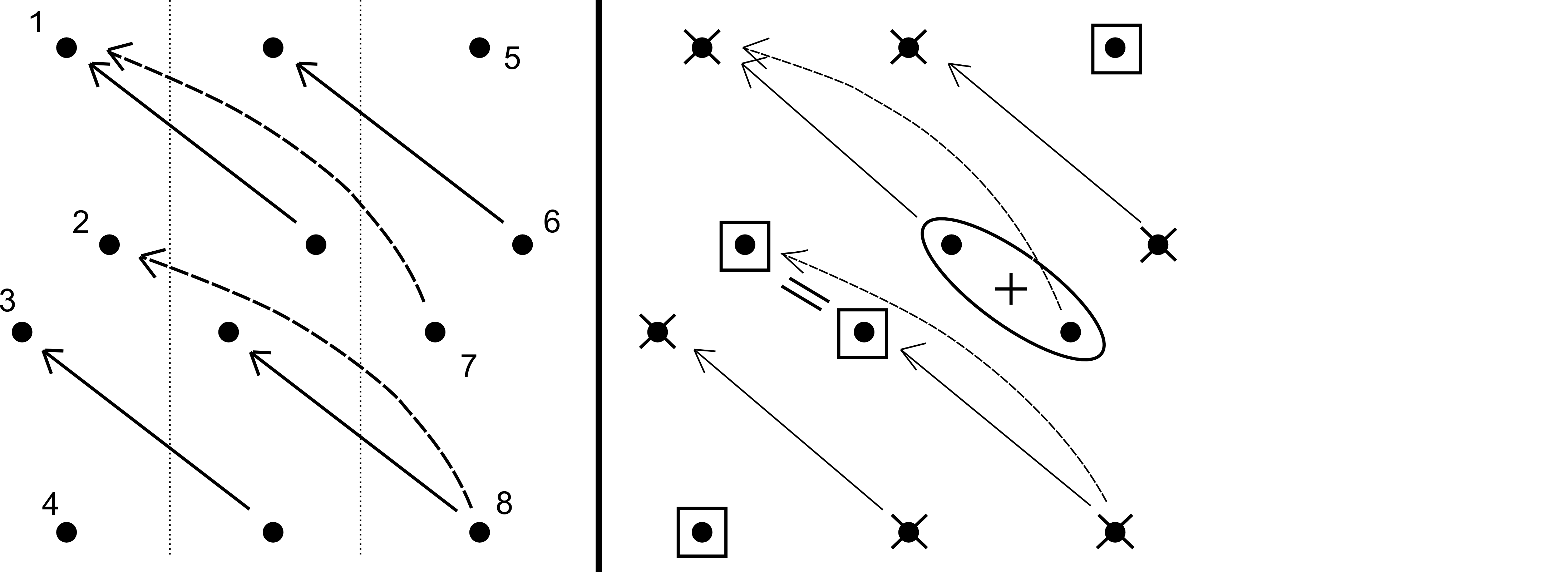}
\caption{ $(R_2 \otimes_A \mathcal{L}_3)_3$: generators for the complex, differentials, and generators for the cohomology. In both diagrams, each dot is a generator for $(R_2 \otimes_A \mathcal{L}_3)_3$; 
the $i^{th}$ column contains the generators corresponding to 
$R_3 \otimes \e \otimes \ldots \otimes \e \otimes \mathcal{L}_3$, with $i-1$ copies of $\e$. 
For instance, suppressing the subscripts 2 and 3: the dot labelled 1 is $r^1 \otimes l^1$, 2 is $r^0 \otimes l^1$, 3 is $r^1 \otimes l^0$, 4 is $r^0 \otimes l^0$; the dot labelled 5 corresponds to $r^1 \otimes \e \otimes \e  \otimes l^1$, and similarly for 6, 7 and 8. The arrows give the differential; the dashed ones come from $\mu^3( \e, \e, \cdot)$ on $\mathcal{L}_3$, and the full ones from $\mu^2(\cdot, \e)$ on $R_2$. The right-hand side diagram gives generators for the cohomology, which has dimension exactly four.
}
\label{fig:j2k3n3}
\end{center}
\end{figure}

\item If $j \leq n <k$, the classes $ r_j^0 \otimes l_k^1 $ and $r_j^1 \otimes \underbrace{\e \ldots \e}_{n-1}  \otimes  l_k^0 $ survive to cohomology. 

\item If $n < j$, the differential on $(R_j \otimes_A \Lk)_n$ is trivial, and the (co)homology has rank $4n$.

\end{itemize}

In all cases, two of the classes are at the `start' of the complex ($r_j^0\otimes l_k^0$ and $r_j^0 \otimes l_k^1$ or its equivalence class), and two at the `end' (e.g.  $r_j^1 \otimes \underbrace{\e \otimes \ldots \otimes \e}_{n-1} \otimes l_k^1$ and $r_j^1 \otimes \underbrace{\e \ldots \e}_{n-1} \otimes l_k^0$, or its equivalence class). 
Suppose additionally that $n \geq 2$ and, and neither $L_0$ nor $L_1$ is quasi-isomorphic to $L$ in the Fukaya category. By Corollary \ref{th:multiplybyenontrivialimage}, $k \geq j \geq 3$; it is then easy to check that there are actually at least 8 classes that survive. The strategy is similar to that in (b) and (b'), and we shall not go through each case in detail; the point is that you can add two `front'-like classes, shifted by an inserted $\e$, and two `back'-like classes, shifted by deleting one of the $(n-1)$ $\e$'s that appear in the expressions for them.
See figures ref{fig:j3k3n4} through \ref{fig:j3k4n5} for examples: the cases where $(j , k, n)$ is equal to $(3, 3, 4)$, $(3,4,4)$ and $(3,4,5)$.

\begin{figure}[htb]
\begin{center}
\includegraphics[height=1.6in, width=5in,angle=0]{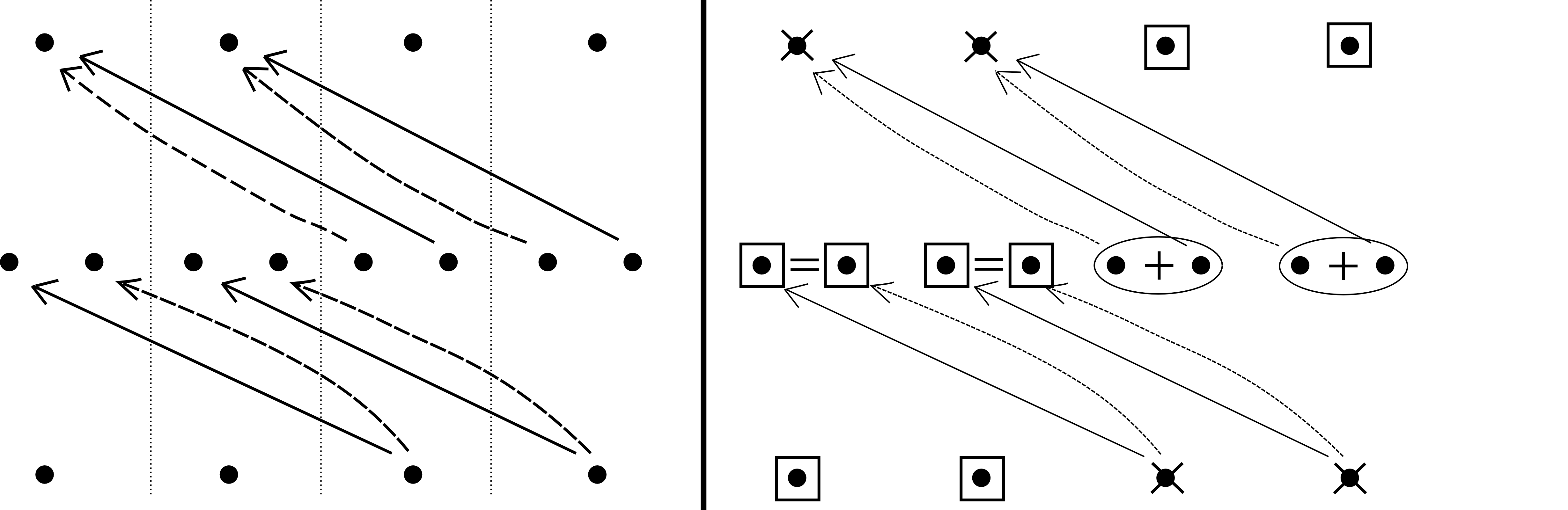}
\caption{ $(R_3 \otimes_A \mathcal{L}_3)_4$: generators for the complex, differentials, and generators for the cohomology. 
}
\label{fig:j3k4n5}
\end{center}
\end{figure}

\begin{figure}[htb]
\begin{center}
\includegraphics[height=1.6in, width=5in,angle=0]{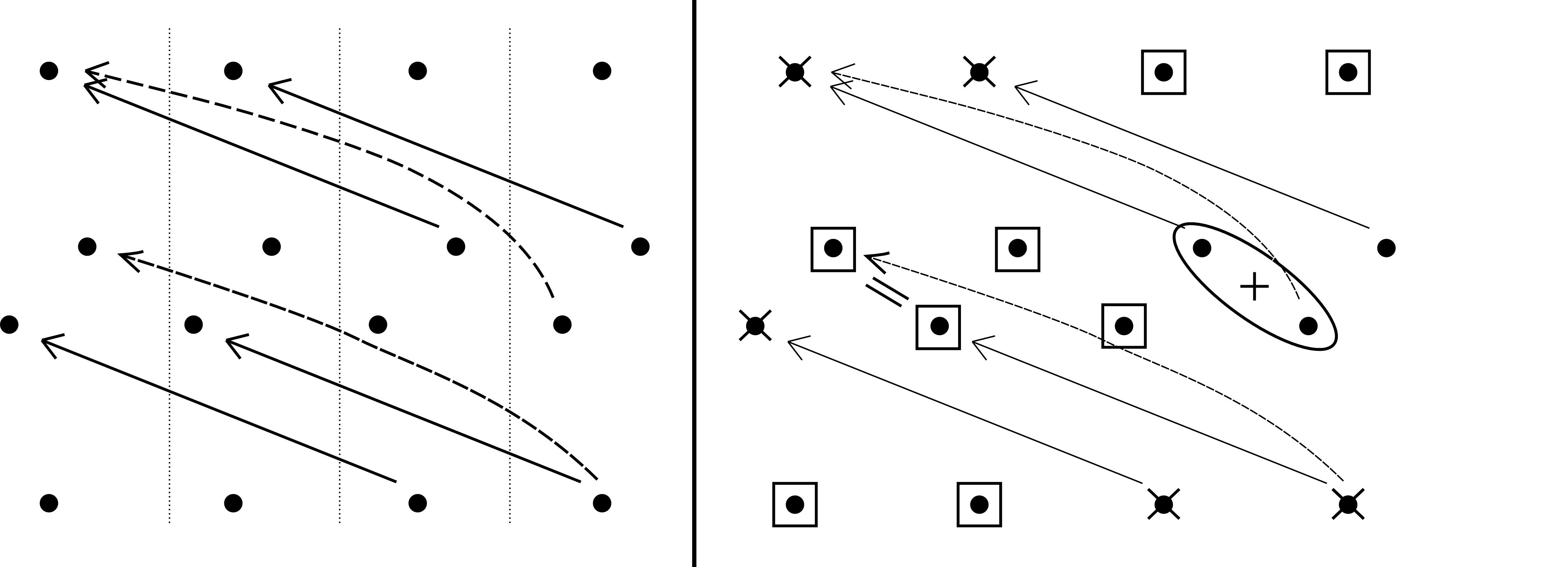}
\caption{$(R_3 \otimes_A \mathcal{L}_4)_4$: differentials, and generators for the cohomology.}
\label{fig:sequence5}
\end{center}
\end{figure}
\end{proof}

\begin{figure}[htb]
\begin{center}
\includegraphics[height=1.6in, width=6in,angle=0]{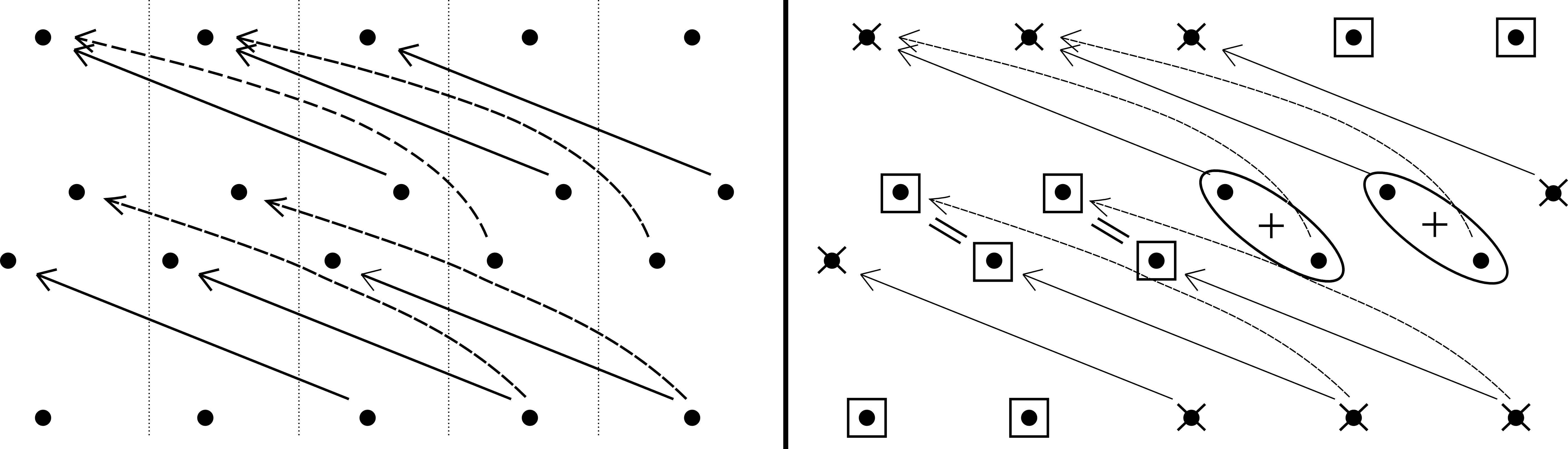}
\caption{ $(R_3 \otimes_A \mathcal{L}_4)_5$: generators for the complex, differentials, and generators for the cohomology. 
}
\label{fig:j3k4n5}
\end{center}
\end{figure}

\begin{remark}\label{rmk:weakerhypothesis2}
Suppose that we started with weaker hypothesis, as in Remark \ref{rmk:weakerhypothesis}; one could proceed analogously to above to get the inequality
\bq
hf (\tau^n_L, L_0, L_1) + hf(L_0,L_1) \geq rk H\big( (hom(L, L_0)\otimes_B hom(L_1, L))_n \big).
\eq
 However, we cannot in general obtain similar lower bounds for the right-hand side.
\end{remark}


\section{Conclusion of argument}\label{sec:conclusion}
 
From here on, the proof of Theorems \ref{th:freegroup} and \ref{th:freegrouponFukaya} closely follows Ishida's argument. Rank of Lagrangian Floer cohomology plays the role of intersection numbers, except we have to be careful whenever it is 2. To get the conclusion of Theorem \ref{th:freegroup}, we will be repeatedly using the fact that if a symplectomorphism $\phi$ of $M$ is symplectically isotopic to the identity $Id$, then for any Lagrangian sphere $L$, $\phi(L)$ is exact Lagrangian isotopic to $L$.

\begin{lem}\label{lem2.3}(see \cite[Lemma 2.3]{Ishida})
Let $L$, $L_0$ and $L_1$ be Lagrangians such that $L$ is a sphere, $L \ncong L_0$ in the Fukaya category, and $hf(L,L_0) \geq 2$. Then for all $n \neq 0$,
\bq
hf(L, L_1) > hf (L_0,L_1) \Rightarrow hf(L, \tau^n_L (L_1)) < hf(L_0, \tau^n_L(L_1))
\eq
\end{lem}

\begin{proof}
Replacing $n$ by $-n$, the statement is equivalent to
\bq
hf(L, L_1) > hf(L_0, L_1) \Rightarrow hf(L, L_1) < hf (\tau^n_L (L_0), L_1)
\eq
As $hf(L_0, L) \geq 2$, from Proposition \ref{th:rankmasterinequalities} we have
\bq
2 hf(L,L_1) - hf(L_0, L_1) \leq hf(\tau^n_L (L_0), L_1)
\eq
which, assuming $hf(L, L_1) - hf(L_0, L_1) > 0$, gives
\bq
hf(L, L_1) < hf (\tau^n_L (L_0), L_1)
\eq
as required.
\end{proof}
We are now ready to conclude the proofs. 
\subsubsection* {Case $hf(L, L') > 2$.}
As $hf(L,L') >  hf(L',L')$, Lemma \ref{lem2.3} immediately implies that we cannot have $\tau_L^n = 1$ for any $n \neq 0$. The same holds for $\tau_{L'}$.
Thus, if the group generated by $\tau_L$ and $\tau_{L'}$ is not free, there must exist $k \in \mathbb{N}$ and $a_i$, $b_i \in \Z^\ast$, $1\leq i \leq n$, such that
\bq
\tau_{L'}^{b_n} \tau_L^{a_n} \ldots \tau_{L'}^{b_1} \tau_L^{a_1} = 1.
\eq
Notice that by assumption, $hf(L', \tau_L^{a_1}L) > hf (L, \tau_L^{a_1}L) = 2.$ By Lemma \ref{lem2.3}, 
\bq
hf(L', \tau_{L'}^{b_1} \tau_L^{a_1}L) < hf(L, \tau_{L'}^{b_1} \tau_L^{a_1}L).
\eq
We may then use Lemma \ref{lem2.3} to see that 
\bq
hf(L', \tau_L^{a_2} \tau_{L'}^{b_1} \tau_L^{a_1}L) > hf(L, \tau_L^{a_2} \tau_{L'}^{b_1} \tau_L^{a_1} L).
\eq
Repeated iterations give
\bq
hf(L', \tau_{L'}^{b_k} \ldots  \tau_L^{a_1} L) < hf(L, \tau_{L'}^{b_k} \ldots \tau_L^{a_1} L)
\eq
where we recognise the left-hand side to be $hf(L',L)$ and the right-hand side to be $hf(L,L)$, a contradiction.

\subsubsection*{Case $hf(L,L') = 2$.}
The key will be:
\begin{cl}\label{th:claim}
For all $m \neq 0$, we have $hf(L', \tau_{L'}^m L) < hf(L, \tau_{L'}^m L)$.
\end{cl}
Notice that this immediately implies that for all $m\neq 0$, $\tau_{L'}^m \neq 1$. Similarly, $\tau_L^m \neq 1$ for all such $m$. Moreover, the statement is equivalent to 
\bq
hf(L', \tau_{L'}^{b_1} \tau_L^{a_1}L) < hf (L, \tau_{L'}^{b_1} \tau_L^{a_1} L)
\eq
for all $a_1\in \Z$, $b_1 \in \Z^\ast$,
which is the second inequality we got in the iteration we conducted for the first case. 
We may then proceed in exactly the same way, repeatedly using Lemma \ref{lem2.3} to get that
\bq
hf(L', \tau_{L'}^{b_k} \ldots  \tau_L^{a_1} L) < hf(L, \tau_{L'}^{b_k} \ldots \tau_L^{a_1} L)
\eq
which is a contradiction.

Thus, to prove that the subgroup generated by $\tau_L$ and $\tau_L'$ is free, it remains only to prove the claim.
\begin{proof}{of claim \ref{th:claim}.}
We treat the cases $m=1$ and $m >1$ separately.
For $m>1$, the second part of Lemma \ref{lem2.3} gives 
\bq
hf(\tau_{L'}^m L, L) \geq 2 hf(L', L)\cdot hf(L,L') - hf(L,L)
\eq
where the right-hand side is equal to 6, so we are done.

For $m=1$, we are certainly done unless $hf(\tau_{L'}L,L)=2$. If this is the case, use the exact sequence:

\begin{equation} \ldots \to {HF(L,L)} \to {HF(L,\tau^{-1}_{L'}L)} \to
 {HF(L',L) \otimes HF(L,L')} \to \ldots
\end{equation}
Considering ranks, this must split as a short exact sequence
\bq
0 \to HF(L, \tau_{L'}^{-1}L) \to
HF(L, L') \otimes HF(L',L) \to
HF(L,L) \to 0.
\eq
In particular, multiplication $HF(L, L') \otimes HF(L',L) \to
HF(L,L)$ is surjective. By Corollary \ref{th:surjectivegiveiso}, $L$ and $L'$ are quasi-isomorphic in the Fukaya category, a contradiction.
\end{proof}


\section{Examples: some Milnor fibres}\label{sec:examples}

We first explain what we mean by a Milnor fibre of a hyperplane singularity, as a symplectic manifold, and establish some relationships between fibres of a fixed singularity (section \ref{sec:symplecticfibre}). Section \ref{sec:adjacency} discusses adjacency of singularities. We then collect material on Khovanov and Seidel's framework for studying fibres of $(A_m)$ singularities, and certain Lagrangian spheres in them (section \ref{sec:KhSeidel}). Subsection \ref{sec:actualexamples} presents the actual examples, constructed using this framework and adjacency arguments. 

\subsection{Milnor fibres as symplectic manifolds}\label{sec:symplecticfibre}

Let $f$ be a non-constant polynomial $\C^{n+1} \to \C$; suppose that $f(0) = 0$, and that $f$ has an isolated critical point at 0.  Assume that $n \geq 2$.

\begin{thm*}(Milnor \cite[Theorems 4.8 and 5.11]{MilnorSingularities})
Fix $f$ as above.
There exists $\e_f$ such that for all $\e < \e_f$,
the sphere $S_\e$ intersects  $ f^{-1}(0)$ transversely, and
 the space $S_\e \backslash f^{-1}(0)$ is a smooth fibre bundle over $S^1$, with projection mapping $$\phi(z) = f(z) / |f(z)|.$$ 
The class of the fibre, as a smooth manifold, is independent of $\e$. Moreover, given $\e$, there exists $\delta(\e)$ such that for all $c$ with $0 \neq |c| < \delta(\e)$, the complex hypersurface $f^{-1}(c)$ intersects $S_\e$ transversely, and $f^{-1}(c) \cap B_\e$, where $B_\e$ is the open $\e$-ball, is a smooth manifold which is diffeomorphic to the fibre $\phi^{-1}(\text{arg}(c))$.
\end{thm*}
Milnor \cite{MilnorSingularities} also shows that $f^{-1}(c) \cap B_\e$ is homotopy equivalent to a finite bouquet of half-dimensional spheres. The count of these is called the Milnor number of the singularity; there is an algebraic expression for it in terms of the partial derivatives of $f$. 

We may assume that $\delta(\e)$ increases monotonely for $\e \in (0, \e_f)$. Let $\d(\e_f)$ be their supremum. We say $c \in \C$ is admissible if $|c| < \delta(\e_f)$. For a given $c$, we define its admissible range to be the interval consisting of all $\e$ such that $c < \delta(\e)$. Note that as $c$ tends to zero, the lower end-point of the admissible range for $c$ does too.

Assume $c$ is admissible, and that $\e$ is in the admissible range of $c$. Define 
\bq
F_{c, \e} := f^{-1}(c) \cap \overline{B}_{\e}
\eq
where $\overline{B}_{\e}$ be the closed ball of radius $\e$. 
When it is clear which $\e$ we are using, we will sometimes denote this simply by $F_c$; if there are several polynomials involved, we will use the notation $F^f_c$ or $F^f_{c, \e}$. 
For any admissible $\e$, there exists $\e' > \e$ that is also admissible; if $p \in \partial F_{c, \e}$, our convention is that $T_p F_{c, \e}$ means $T_p F_{c, \e'}$. 
Let $\theta = i /4 \sum_{i = 0}^n (z_i d \bar{z}_i - \bar{z}_i dz_i)$ and $\Omega = d \theta$ be the usual forms on $\C^{n+1}$. 
Let $Z$ be the associated negative Liouville vector field: $i_Z \Omega = -\theta$, and $\phi_t$ its flow, where $t$ parametrizes time.

The restrictions of $\theta$ and $\Omega$, say $\theta_c$ and  $\omega_c$, give each $F_{c, \e}$ the structure of an exact symplectic manifold. We shall call any such manifold the`Milnor fibre of $f$'.
Let $Z_c$ be the negative Liouville vector field on $F_{c, \e}$. Notice $- Z_c$ is the gradient of $\sum_{i=0}^n |z_i|^2$ with respect to the usual K\"ahler metric; this points outwards along level-sets. Thus, as  $f^{-1}(c) \pitchfork S_{\e}$, $Z_c$ points inwards along the boundary of $F_{c, \e}$. 
We denote the negative Liouville flow on each fibre by $\phi^{\fib}_t$.

Following a well-established approach for exact symplectic manifolds with contact type boundary, we glue cylindrical ends to $F_{c, \e}$ to obtain a non-compact symplectic manifold $\overline{F}_{c, \e}$:
\bq
(\overline{F}_{c, \e}, \overline{\omega}_c, \overline{\theta}_c) = (F_{c, \e}, \omega_c, \theta_c) \cup (\partial F_{c, \e} \times \R^+ , d(e^t \theta_c |_{\partial F_{c, \e}}), e^t \theta_c |_{\partial F_{c, \e}})
\eq
where $t$ is the coordinate on $\R^+$, and the gluing is made using the negative Liouville flow $\phi^{fib}_t$.  
For a fixed $c$, this construction is independent of the choice of admissible $\e$; hereafter we will often suppress the subscript $\e$.

\begin{lemma}\label{th:changingc}
Suppose $c_1$ and $c_2$ are admissible for $f$. Then there is an exact symplectomorphism between  $\overline{F}_{c_1}$ and $\overline{F}_{c_2}$. 
More precisely, suppose that $\e$ is any constant in $(0, \e_f)$ that lies in the admissible range for both $c_1$ and $c_2$ (notice that we can always find such an $\e$). To any smooth path in $B_{\d(\e)} \backslash 0$, with end-points at $c_1$ and $c_2$, we associate a (non-canonical) exact symplectomorphism of the completed fibres, defined up to Hamiltonian isotopy.
\end{lemma}

Fix $\e < \e_f$, and let \bq E = \bigcup_{|c| < \delta(\e)} F_{c, \e}. \eq
Let $\partial E \subset E$ be its `horizontal' boundary, i.e. the union of the boundaries of all the $F_{c, \e}$. (When talking about the tangent space to $E$ at a point of $\partial E$, we use the same convention as for $F_{c, \e}$.) Putting together the negative Liouville flows on each fibre gives smooth maps $E \to E$, that we shall also denote $\phi^{\fib}_t$; the associated vector field ($Z^c$ on each fibre) will be called $Z^{\fib}$. 

Fix $c$ with $|c| < \d(\e)$. For every $p \in F_c$, the vector space $T_p F_c \subset T_p E$ has a canonical complement, given by taking the symplectic orthogonal to $T_p F_c$. This gives a `horizontal tangent space' for our fibration: every vector in $T_{f(p)} B_{\d(\e)}$ has a preferred lift in $T_p E$. 
Fix a smooth path $ \gamma: [0, 1] \to B_{\d(\e)} \backslash 0$. The `symplectic parallel transport' associated to the horizontal tangent space is a priori only defined on a (possibly empty) subset of $F_{\gamma(0)}$: at points of the boundary $\partial E$, the horizontal tangent space may not lie in the tangent space of $\partial E$.  By construction, we have the following:
\begin{lemma}
Symplectic parallel transport defines an exact symplectomorphism from its domain in $F_{\gamma(0)}$ to its image in $F_{\gamma(1)}$. 
\end{lemma}
Note that we could carry out this process for any symplectic form on $E$ that restricts to one on the fibres, and that it is enough for the form to be closed and non-degenerate on each fibre. 

Considering volumes, one cannot hope in general for symplectomorphisms defined on the whole of each fibre. Instead, we shall work with the completed fibres. The key is:

\begin{lemma}\label{th:nicehorizontaltangentspace}
We can construct a closed form $\Omega''$ on $E$ such that
\begin{enumerate}
\item for every $c$, $\Omega'' |_{F_c} = \Omega|_{F_c}$;

\item the horizontal tangent space determined by $\Omega''$ is invariant under $\phi^{\fib}_t$ in some collar neighbourhood of the boundary $\partial E$;

\item $\Omega''$ and $\Omega$ agree outside a collar neighbournood of $\partial E$.
\end{enumerate}
\end{lemma}
To prove Lemma \ref{th:changingc}, let us replace $\Omega$ by such an $\Omega''$, and work instead with the horizontal tangent space given by $\Omega''$. Property 2 then allows us to extend the corresponding parallel transport to the completed fibres, in the obvious fashion. 

\begin{proof}of Lemma \ref{th:nicehorizontaltangentspace}.
There exists $\tau > 0$ such that $Z^{\fib}$ is non zero at all points of
\bq
U  = \bigcup_{t \in [0, \tau] } \phi^{\fib}_t (\partial E) \,\, \subset E.
\eq
We will use the identification
\begin{equation}
\begin{array}{ccl}
\partial E \times [- \tau, 0] & \cong & U  \\
 (a, t) & \mapsto & \phi^{\fib}_{-t}(a).
\end{array}
\end{equation}
Let $\alpha = \theta|_{\partial E}$. Define $\theta' \in \Omega^1(U)$ by
\bq
\theta'= e^t (\phi_t^{\fib})^{\ast} \alpha
\eq
where $t \in [-\tau, 0]$. By construction, $\theta'|_{F_c} = \theta|_{F_c}$, for all fibres $F_c$. The form $\Omega' = d \theta' \in \Omega^2 (U)$ satisfies conditions one and two.

Let $\xi = \theta' - \theta \in \Omega^1 (U)$. Let $\psi$ be  a smooth cut-off function on $[-\tau, 0]$ such that $\psi = 1$ on $[-\tau/2, 0]$ and $\psi =0 $ on $[-\tau, -2\tau /3]$. This induces a function $U \to \R$, that we also denote by $\psi$. Set 
\bq
\Omega'' = \Omega + d (\psi \xi) \in \Omega^2(U).
\eq
This agrees with $\Omega'$ on an collar neighbourhood on $\partial E$, and with $\Omega$ outside a (larger) collar neighbourhood of $\partial E$, and so satisfies conditions 2 and 3. As $\xi$ vanishes on fibres, condition one is also satisfied.
\end{proof}


\paragraph{Holomorphic reparametrization.}

So far, we have treated $f$ simply as the germ at the origin of a function $\C^{n+1} \to \C$; from the point of view of singularity theory, it is more natural to think of $f$ as a representative of its equivalence class under biholomorphic change of coordinates (preserving the origin). What can we say for the corresponding (completed) fibres? 

\begin{lemma}\label{th:holoreparemetrisation}
Say $f = g \circ h$, some holomorphic change of coordinates $h$. 
Then there is an exact symplectic embedding from  a Milnor fibre of $f$ to a completed Milnor fibre of $g$, and vice-versa.
\end{lemma}

\begin{proof}
First note that $h$ maps (subsets of) Milnor fibres of $f$ to (subsets of) Milnor fibres of $g$. 
There exists $\e'_f$ such that $h(B_{\e'_f}) \subset B_{\e_g}$. Fix $c$ such that $ |c| < \d(\e'_f)$. $h$ maps $F^f_c$ into $F^g_c$, and thus gives an embedding $F_{c, f} \hookrightarrow \overline{F}_{c,g}$; this is not in general symplectic. However, notice that both symplectic forms are compatible with $J$. By connecting the associated metrics, we can find a path connecting the symplectic forms. Also, note that any closed two-form is automatically exact. Thus we can use a Moser argument to deform our embedding to a symplectic one. By homology considerations, this symplectic embedding is then automatically exact (notice that we are using $n \geq 2$). 
\end{proof}

\begin{remark}
A more careful set-up might give exact symplectomorphisms between completed Milnor fibres of $f$ and $g$. The above will suffice for Corollary \ref{th:propetySforanyfibre}.
\end{remark}


\paragraph{Property $S$.} 

We say that a exact symplectic manifold with contact type boundary has `property $S$' if it satisfies the hypothesis of Theorem \ref{th:freegroup}: its interior contains two Lagrangian spheres $L_0$, $L_1$ such that $hf(L_0, L_1) \geq 2$ and $L_0$ and $L_1$ are not Fukaya isomorphic.

\begin{lemma}\label{th:propertyS}
Suppose there's an exact symplectic embedding from a Milnor fibre $F$ to a completed Milnor fibre $\overline{F'}$ (these needn't be fibres of the same polynomial), and, moreover, that $F$ has property $S$.  Then $F'$ has property $S$. 
\end{lemma}

\begin{proof}
Let $\iota: F \hookrightarrow \overline{F'}$ be the exact symplectic embedding, and $L_0$, $L_1 \subset F$ the Lagrangian spheres given by property $S$. 
Let $\overline{F'}_T = F' \cup \partial F' \times [0, T]$ be a truncation of the completed Milnor fibre; this is an exact symplectic manifold with contact type boundary. Choose $T$ large enough such that $\iota(F) \subset int(\overline{F'}_T)$. 
If follows from Lemma \ref{th:Abouzaid} that $hf(\iota(L_0), \iota(L_1)) = hf(L, L')$, and $\iota(L_0)$ and $\iota(L_1)$ are not isomorphic in $\Fuk (\overline{F'}_T)$. 
Flow $\iota(L_0)$ and $\iota(L_1)$ by the negative Liouville vector field inside $\overline{F'}_T$ until their images, say $L'_0$ and $L'_1$, lie in the interior of $F'$. $L'_0$ is quasi-isomorphic to $L_0$ in $\Fuk(\overline{F'}_T)$, and similarly for $L'_1$ and $L_1$. Thus $L'_0$ and $L'_1$ are not quasi-isomorphic, and $hf(L'_0, L'_1) = hf(L_0, L_1)$. Using Lemma \ref{th:Abouzaid} again, we see that this is also true in the Fukaya category of $F'$.
\end{proof}

The following is now immediate.

\begin{cor}\label{th:propetySforanyfibre}
If $F_c$, any Milnor fibre of a polynomial $f$, has property $S$, then so do all completed Milnor fibres of $f$. Moreover, if $g$ is another representative of the equivalence class of the germ of $f$, then all Milnor fibres of $g$ have property $S$.
\end{cor}

\subsection{Adjacency of singularities and symplectic embeddings}\label{sec:adjacency}

Hereafter an `isolated singularity' $[f]$ will be the germ of a polynomial $f: \C^{n+1} \to \C$, with $f(0) = 0$ and an isolated singularity at $0$, up to biholomorphic change of coordinates (fixing the origin).

\begin{definition}
Let $[f]$ and $[g]$ be isolated singularities. $[f]$ is said to be adjacent to $[g]$ if there exist arbitrarily small polynomial perturbations $p_k$, $k \in \mathbb{N}$, such that $[f + p_k ] = [g]$. 
\end{definition}
Adjacency of singularities is a well-studied classical topic; we shall use material collected in the survey \cite{Arnold6}. 

\begin{lemma}\label{th:adjacencyfibresembed}
Suppose that $[f]$ and $[g]$ are isolated singularities such that $[f]$ is adjacent to $[g]$, and that $n \geq 2$. Then there exists an exact symplectic embedding from a Milnor fibre of $g$ to a completed Milnor fibre of $f$.
\end{lemma}

\begin{proof}
Consider the map
\begin{equation}
\begin{array}{crclcrcl}
q: & \C^{n+1} \times \C & \to  & \C \times\C \\
& q (\mathbf{z} , \, \gamma) & = & ( f(\mathbf{z}) + \, \gamma p_k(\mathbf{z}), \gamma )
\end{array}
\end{equation}
There exists an $n \in \mathbb{N}$ and $\e < \e_f$ such that for all $(\mathbf{z}, \gamma) \in B_{\d(\e)} \times B_2$, $q^{-1}(c, \gamma) \pitchfork S_{\e}$, where $S_\e$ is the sphere in $\C^{n+1} \times {\gamma}$.
Let $Q_{(c, \gamma)} = q^{-1}(c , \gamma) \cap \overline{B}_\e$. Let $H$ be the set of points $(c, \gamma)$ in $B_{\d(\e)} \times B_2$ such that $Q_{(c, \gamma)}$ is singular. $H$ is an algebraic curve; in particular, this means that $(B_{\d(\e)} \times B_2 )\backslash H$ is path-connected. The $Q_{(c, \gamma)}$ inherit an exact symplectic structure from $\C^{n+2}$, and can be extended to non-compact symplectic manifolds $\overline{Q}_{(c, \gamma)}$ by gluing conical ends, in a similar fashion to the $\overline{F}_c$. For sufficiently small $c$, $Q_{(c, 1)}$ contains a Milnor fibre of a representative of $[g]$ as a subset: if $[f] \neq [g]$, we certainly have $\e_{f+p_k} \leq e$. $Q_{(c', 0)}$ is a Milnor fibre of $f$. Pick a smooth path in $(B_{\d(\e)} \times B_2 )\backslash H$ connecting $(c,1)$ and $(c',0)$.
Adapting the construction used to prove Lemma \ref{th:changingc}, we get exact symplectomorphisms from $\overline{Q}_{(c,1)}$ to $\overline{Q}_{(c', 0)}$, completing the proof.

\end{proof}

\begin{cor}\label{th:propertySadjacency}
Suppose that $[f]$ is adjacent to $[g]$, and that $[g]$ has property $S$. Then $[f]$ also has property $S$.
\end{cor}

\begin{proof}
This follows immediately from Lemmata \ref{th:propertyS} and \ref{th:adjacencyfibresembed}.
\end{proof}


\subsection{Lagrangian spheres in $(A_m)$ fibres}\label{sec:KhSeidel}

The singularity of type $(A_m)$ is the one associated to the polynomial $z_0^2 + \ldots + z_n^2 +z_{n+1}^{m+1}$; we are interested in the case $n \geq 2$. 
For any $m$, \cite{KhovanovSeidel} describes how to associate Lagrangian submanifolds of the $(A_m)$  fibre to certain curves on the unit disc with $(m+1)$ marked points, which we denote by $D_{m+1}$. If the curve intersects the marked points exactly twice, once at both endpoints, one gets a Lagrangian sphere, defined up to Lagrangian isotopy. Also, an isotopy of the curve, relative to the marked points, gives a Lagrangian isotopy of the sphere.

Suppose the marked points are aligned horizontally. A basis for the homology of the fibre is given by the spheres corresponding to the straight-line segments between consecutive marked points. (See figure \ref{fig:basis} for  the case $m=2$.)
\begin{figure}[htb]
\begin{center}
\includegraphics[height=1.5in, width=4in,angle=0]{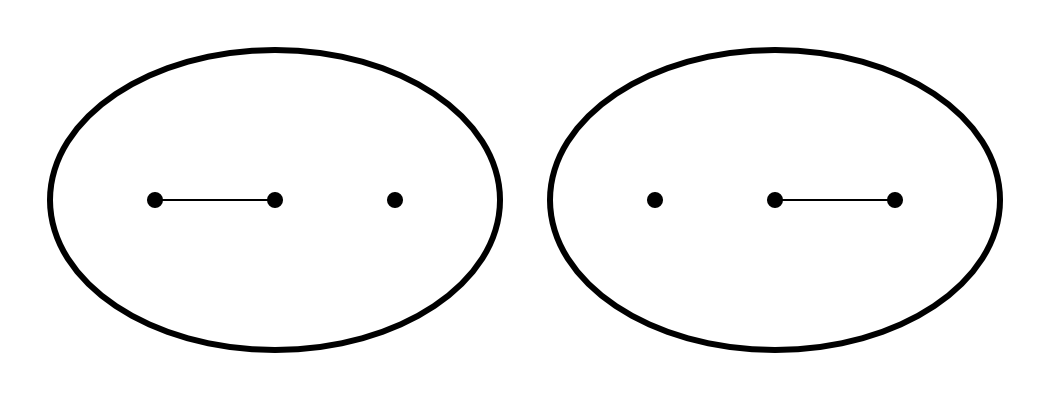}
\caption{Curves corresponding to a basis for homology for the $(A_2)$ fibre, say $a$ and $b$.}
\label{fig:basis}
\end{center}
\end{figure}

A lot of information about these spheres can be read straight from the curves in $D_{m+1}$. Notably, the rank of the Floer cohomology between any two such spheres is twice the geometric intersection number of the curves they are are associated to \cite[Theorem 1.3]{KhovanovSeidel}. For two curves that intersect `minimally' (see \cite{KhovanovSeidel}), this number counts $1/2$ for a common end-point, and 1 for all other intersection points.

 It is also easy to determine the homology class of a sphere. 
For even $n$, it only depends on the end-points of the corresponding curve, and orientation; any two curves with common end-points actually give smoothly isotopic spheres \cite[Proposition 5.1]{Maydanskiy}. Given a sphere associated to a curve in $D_{m+1}$, the action of the Dehn twist about that sphere can be described in terms of a half-twist of $D_{m+1}$, preserving the marked points \cite[Lemma 7.1]{MaydanskiySeidel}. To compute homology classes of spheres when $n$ is odd, one needs to use this together with the Picard-Lefschetz theorem (\cite{Picard,Lefschetz}; see e.g.~\cite[section 2.1]{Arnold6} for a concise account). This
applies to all the Lagrangian spheres considered here; for a fixed sphere $S$, it describes the action of the Dehn twist $\tau_S$ on a class $x$ in the middle homology of the Milnor fibre:
\begin{equation}
\tau_S(x) = x + (-1)^{n(n+1)/2}(x \circ S)S.
\end{equation}

\subsection{Some families of examples}\label{sec:actualexamples}

\subsubsection{Even complex dimension}

There are many examples of pairs of Lagrangian spheres in the $(A_2)$ fibre realising property $S$ that can be constructed using the set-up of \cite{KhovanovSeidel}. Moreover, one can arrange for them to be smoothly isotopic by using the fact that it is enough for them to share end-points. 

\begin{figure}[htb]
\begin{center}
\includegraphics[height=1.5in, width=4in,angle=0]{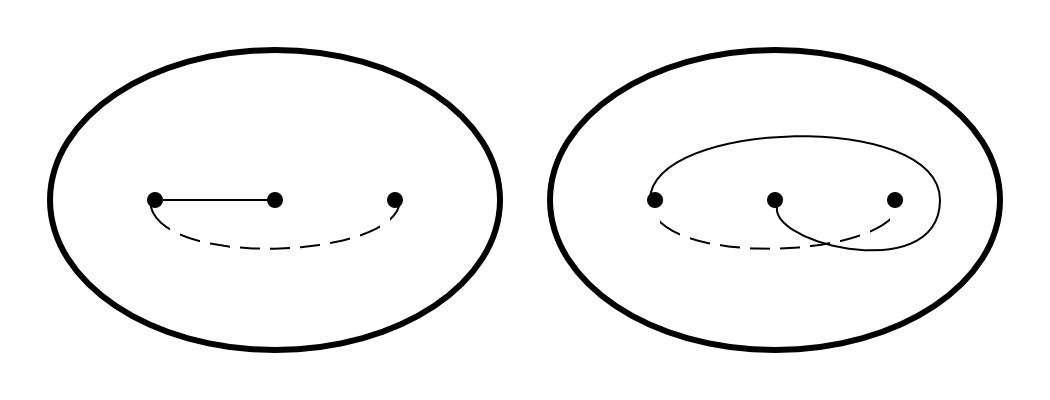}
\caption{Even case: curves corresponding to two homologous spheres in the $(A_2)$ fibre.}
\label{fig:A_2case1}
\end{center}
\end{figure}

Arguably the simplest example is given in figure \ref{fig:A_2case1}; the two full-stroke curves correspond to two isotopic spheres, say $L_0$ and $L_1$. The corresponding curves only intersect at their end-points, so by \cite[Theorem 1.3]{KhovanovSeidel},
$ hf(L_0, L_1) = 2 $.
Consider the sphere associated to the dashed curves, say $L_2$. We have $hf(L_0, L_2) = 1$ and $hf(L_1, L_2) = 3$; thus $L_0$ and $L_1$ are not Fukaya isomorphic. They satisfy the hypothesis of Theorem \ref{th:freegroup}: $(A_2)$ has property $S$.

\begin{lem}
Any degenerate singularity is adjacent to $(A_2)$.
\end{lem}

\begin{proof}
By the parametric Morse lemma (see e.g.~\cite[Section 1.3]{Arnold6}), a degenerate singularity is equivalent to a function of the form 
\begin{equation}
\phi(z_0, \ldots, z_k) + z_{k+1}^2 + \ldots + z_{n}^2
\end{equation}
where $\phi \in \mathcal{M}^3$, and $\mathcal{M} \subset \C[z_0, z_1, \ldots, z_n]$ is the maximal ideal. For sufficiently small $\e >0$, 
\begin{equation}
\phi(z_0, \ldots, z_k) +\e(z_1^2 + z_2^2+ \ldots + z_k^2) + z_{k+1}^2 + \ldots + z_{n}^2
\end{equation}
has an isolated singularity at zero. Moreover, by construction, it has corank one. The parametric Morse lemma implies that it is an $(A_m)$ singularity, for some $m \geq 3$; in particular, it is adjacent to $(A_3)$ \cite[Section 2.7]{Arnold6}.
\end{proof}

 Thus, by Corollary \ref{th:propertySadjacency}, the Milnor fibre $F$ of a degenerate singularity has property $S$; moreover, the two Lagrangian spheres realising this can be chosen to be smoothly isotopic, as they come from smoothly isotopic spheres in $(A_2)$. This implies that the Dehn twists about them agree as elements of $\pi_0(\Diff^+F)$.

\subsubsection{Odd complex dimension}

One can still find examples of spheres in the $(A_2)$ fibre that satisfy the hypothesis of Theorem \ref{th:freegroup}, e.g. those above. (The calculation of the rank of Floer cohomology is independent of dimension.) 
However, we cannot find examples where they are also homologous. 

\begin{lem}\label{th:A_2nothomologous}
If two curves in $D_3$ correspond to homologous spheres in the $(A_2)$ fibre, then they are isotopic relative to the marked points. 
\end{lem}

\begin{proof}
Start, exceptionally, by considering the case where the fibre is one-dimensional.  
Let $\CC_0$ and $\CC_1$ be the two curves in $D_3$, and $s_0$ and $s_1$ be the corresponding circles.
For this case, one could carry out a `by hand' proof, starting by analysing the contributions of various portions of the curve to the homology classes corresponding to $a$ and $b$ (as introduced in Figure \ref{fig:basis}).

Alternatively, let us make use of some classical results (see e.g.~\cite{BirmanHilden, Birman}). The $(A_2)$ fibre, a once-punctured torus, is a double cover of $D_3$, branched over the marked points. Denote it by $A^1_2$. Let $\Diff^+D^3$ be the group of orientation-preserving diffeomorphisms of the disc that preserve the three marked points; $\Diff^+_cD^3$ the subgroup of such diffeomorphisms that are compactly supported; similarly with  $\Diff^+A^1_2$ and $\Diff^+_cA^1_2$. Every element of $\Diff^+ D_3$ lifts to two possible elements of $\Diff^+A^1_2$, differing by the deck action; every element of $\Diff^+_c D_3$ induces a unique element of $\Diff^+_c A^1_2$.

The half-twists in $a$ and $b$ are generators for an action of the braid group $Br_3$ on the base $D_3$; this lifts to an action on the total space, and the half-twists give the Dehn twists in the circles corresponding to $a$ and $b$. We have

\bq \label{eq:isotopyclasses}
\xymatrix{
 PSL_2(\Z) \ar[r]^-{\cong} &\pi_0(\Diff^+D^3) 
& \ar[l]_-{2:1} \pi_0(\Diff^+A^1_2) \ar[r]^-{\cong} 
& \pi_0(\Diff^+T^2) \cong  SL_2(\Z)  
\\
Br_3 \ar[r]^-{\cong} 
& \pi_0(\Diff^+_cD^3) \ar@{->>}[u] \ar[r]^-{\cong} 
& \pi_0(\Diff^+_cA^1_2) \ar@{->>}[u] 
& 
}
\eq
where one can use the presentations
\begin{eqnarray}
Br_3 & \cong &
\langle \s_1, \, \s_2 \, |  \, \s _1 \s_2 \s_1 = \s_2 \s_1 \s_2 \rangle \\
 PSL_2(\Z) & \cong & 
\langle \s_1, \, \s_2 \, |  \, \s _1 \s_2 \s_1 = \s_2 \s_1 \s_2, \, (\s_1 \s_2)^3 =1 \rangle \\
 SL_2(\Z)  & \cong & 
\langle \s_1, \, \s_2 \, |  \, \s _1 \s_2 \s_1 = \s_2 \s_1 \s_2, \, (\s_1 \s_2)^6 =1 \rangle.
\end{eqnarray}
(See, for instance, the book \cite{Birman}. Note that a $\Z_2$ equivariant diffeomorphism which is isotopic to the identity is $\Z_2$ equivariantly isotopic to the identity \cite[Theorem 1]{BirmanHilden}.) In particular, the action of $SL_2(\Z)$ on the homology of $A^1_2$ can be realised by $\Z_2$ equivariant maps.

The curves $C_0$ and $C_1$ can both be obtained from $a$ (or $b$) through a series of forwards or backwards half-twists in $a$ and $b$. Choose a map $f \in \Diff^+D^3$ mapping $C_0$ to $C_1$. This lifts to a $\Z_2$ equivariant $\tilde{f} \in \Diff^+(A^1_2)$, mapping $s_0$ to $s_1$. Now pick a $\Z_2$ equivariant $h \in \Diff^+ A^1_2$ that fixes $s_1$, and such that $h \circ \tilde{f}$ acts as the identity on homology (remember that we are assuming that $s_0$ and $s_1$ are isotopic). Any element of $\pi_0(\Diff^+T^2)$ is uniquely determined by its action on homology; thus $h \circ \tilde{f}$ is isotopic to the identity. By construction, it is a $\Z_2$ equivariant map; by \cite[Theorem 1]{BirmanHilden}, it is $\Z_2$ equivariantly isotopic to the identity. This implies that $C_0$ and $C_1$ are isotopic through elements of $\Diff^+(D_3)$. 

Now suppose the fibre, say $A_2^n$, has odd dimension $n \geq 3$. Fix a curve $\CC$ in $D_3$ with boundary on the marked points, and interior disjoint from them. Let $S_{\CC}$ be the sphere corresponding to $\CC$; say $[S_{\CC}] = (S_a, S_b) \in H_n (A_2^n)$, with respect to the basis corresponding to $a$ and $b$. One can also consider the circle $s_{\CC} \subset A_2^1$ associated to $\CC$, as discussed in the previous paragraph; say $[s_{\CC}] = (s_a, s_b) \in H_1(A_2^1)$, also with respect to the basis corresponding to $a$ and $b$. $\CC$ can be obtained from $a$ (or $b$) through a series of forward and backward half-twists in $a$ and $b$; thus the Picard-Lefschetz theorem implies that
\bq
(S_a, S_b) =  (\pm s_a, \pm s_b).
\eq
Reversing some orientations if needed, this allows us to conclude using the $n=1$ case.
\end{proof}

By \cite[Theorem 1.3]{AbouzaidSmith}, every Lagrangian sphere in $A^n_2$, for $n \geq 3$, as an element of the Fukaya category, is in the braid group orbit of (either of) the Lagrangian spheres corresponding to $a$ or $b$. Thus we expect Lemma \ref{th:A_2nothomologous} to imply that there can be no pair of \emph{homologous} spheres in the $(A_2)$ fibre satisfying the hypothesis of Theorem \ref{th:freegroup}.

\begin{figure}[htb]
\begin{center}
\includegraphics[height=1.5in, width=4in,angle=0]{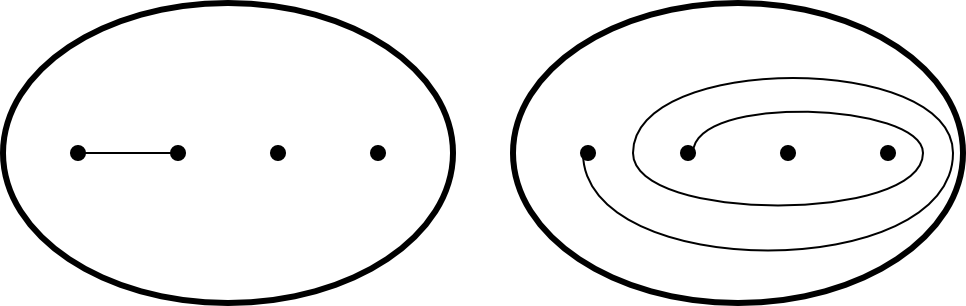}
\caption{Odd case: curves corresponding to two homologous spheres in the $(A_3)$ fibre.}
\label{fig:A_3case2}
\end{center}
\end{figure}

One can nonetheless find pairs of homologous spheres in the $(A_3)$ fibre satisfying the hypothesis of Theorem \ref{th:freegroup}. Figure \ref{fig:A_3case2} gives an example. Let $L_0$ (resp.~$L_1$) be the sphere corresponding to the left (resp.~right)-hand curve. We have $hf(L_0, L_1) = 4$. 
Let $\tau_a$, $\tau_b$ and $\tau_c$ be the Dehn twists in the spheres corresponding to the first, second and third straight-line segments between marked points, similarly to the $(A_2)$ case and figure \ref{fig:basis}. (Note $\tau_a = \tau_{L_0}$.) One can check that 
\bq L_1 = \tau_b \, \tau_c^2 \, \tau_b^2 \, \tau_c^2 \, \tau_b \, L_0 \eq and, using the Picard-Lefschetz theorem, that $L_0$ and $\pm L_1$ are homologous (the sign depends on the $\mathit{mod}$ 4 remainder of $n$). As the $L_i$ have odd real dimension, this implies that $[L_0] \cdot [L_1] = 0$. Moreover, notice that the fibre has real dimension $\geq 6$, and is simply connected. Thus, by the Whitney trick, there is a smooth isotopy of the fibre such that the image of $L_0$ is disjoint from $L_1$. This means that $\tau_{L_0}$ and $\tau_{L_1}$ commute as elements of $\pi_0(\Diff^+)$. 

The same clearly applies to fibres of singularities that are adjacent to $(A_3)$.
What singularites are adjacent to $A_3$? Let us give a family of examples.

\begin{lemma}
Let $\mathcal{M} \subset \C[z_1, z_2]$ be the maximal ideal, generated by $z_1$ and $z_2$, and $f(z_1, z_2)$ be a non-zero element of $\mathcal{M}^4$. Then $f$ is adjacent to $(A_3)$. 
\end{lemma}

\begin{proof}
Consider $f_\beta (z_1, z_2) = f(z_1, z_2) + \beta z_1^2$, where $\beta >0$ is an arbitrarily small constant. Considering Hessians, this is a corank-one singularity; the parametric Morse lemma (see e.g.~\cite[Section 1.3]{Arnold6}) implies it is an $(A_m)$ singularity, for some $m$. As $(A_{k+1})$ is adjacent to $(A_k)$, for all $k$, it is enough to show that $m \geq 3$. To do so, we estimate the Milnor number $\mu$ of $f_\beta$. The ideal $\mathcal{I}$ generated by its partial derivatives is contained in 
\bq
\mathcal{M}^3 + \langle z_1 \rangle = \langle z_1, \, z_2^3 \rangle 
\eq
and thus
\bq
\mu  = rk \big( \C[z_1, z_2] / \mathcal{I} \big) \geq rk\big( \C[z_1, z_2] / \langle z_1, \, z_2^3 \rangle \big) = 3.
\eq
\end{proof}
This implies that the stabilizations of the polynomials $f \in \mathcal{M}^4$, of the form $f(z_1, z_2) + z_3^2 + \ldots + z_{n+1}^2$, are adjacent to the $(A_3)$ singularity in the appropriate dimension. 

\begin{remark}
The reader might wonder whether one could find two Lagrangian spheres $L$, $L'$ in an odd complex-dimensional $(A_2)$ fibre with $[L]\cdot [L'] = 0$, despite Lemma \ref{th:A_2nothomologous}, as one could then imitate the Whitney trick argument above to smoothly isotope them apart. Suppose $L$ and $L'$ are two such Lagrangian spheres. Considering the intersection matrix on homology, we must have $k[L] = m[L']$, some integers $k$, $m$. By \cite[Corollary 1.5]{AbouzaidSmith}, we have $k, m = \pm 1$. Thus $L$ and $L'$ are homologous, and Lemma \ref{th:A_2nothomologous} and the remark thereafter apply. 
\end{remark}


\end{document}